\documentclass[ssy, preprint]{imsart}

\usepackage{amsthm,amsmath}
\usepackage[psamsfonts]{amssymb}
\usepackage[colorlinks,citecolor=blue,urlcolor=blue]{hyperref}
\usepackage{enumerate}
\usepackage[numbers]{natbib}

\arxiv{arXiv:1206.3390}

\startlocaldefs
\newtheorem{theorem}{Theorem}
\newtheorem{lemma}{Lemma}
\newtheorem{proposition}{Proposition}

\newtheorem{assumption}{Assumption}
\newtheorem{definition}{Definition}

\newtheorem*{theorem*}{Theorem}
\theoremstyle{remark}
\newtheorem{remark}{Remark}

\usepackage{xspace}
\newcommand{\xmath}[1]{\ensuremath{#1}\xspace}
\newcommand{\E}{\xmath{\mathbb{E}}}
\renewcommand{\Pr}{\xmath{\mathbb{P}}}
\newcommand{\evA}{\xmath{ S_n > b}}
\newcommand{\pn}{\xmath{n\bar{F}(b)}}
\newcommand{\Et}{\xmath{\mathbb{E}_1}}
\newcommand{\Pt}{\xmath{\mathbb{P}_1}}

\newcommand{\Eh}{\xmath{\mathbb{E}_{1}}}
\newcommand{\Ph}{\xmath{\mathbb{P}_{1}}}
\newcommand{\Eth}{\xmath{\mathbb{E}_{2}}}
\newcommand{\Pth}{\xmath{\mathbb{P}_{2}}}

\newcommand{\evk}{\xmath{n_{k-1} < \tau_b \leq n_k}}
\newcommand{\Fi}{\xmath{\bar{F}_I}}
\newcommand{\nkm}{\xmath{n_{k-1}}}
\newcommand{\pK}{\xmath{p_{_K}}}
\newcommand{\tn}{\xmath{\theta _{n}}}
\newcommand{\tnb}{\xmath{\theta _{n,b}}}
\newcommand{\gn}{\xmath{\gamma / \phi_n}}
\newcommand{\fin}{\xmath{\phi _n}}
\newcommand{\ega}{\xmath{e^{\gamma }}}

\newcommand{\thr}{\xmath{n^{\beta +\epsilon}}}
\newcommand{\Ad}{\xmath{A_{\textnormal{dom}}}}
\newcommand{\Ar}{\xmath{A_{\textnormal{res}}}}
\newcommand{\Zd}{\xmath{Z_{\textnormal{dom}}}}
\newcommand{\Zr}{\xmath{Z_{\textnormal{res}}}}
\newcommand{\Gb}{\xmath{\bar{G}^{(\beta)}}}
\newcommand{\Gib}{\xmath{\bar{G}^{(\beta)}_I}}
\endlocaldefs

\begin{document}

\begin{frontmatter}

\title{State-independent Importance Sampling for Random Walks with
  Regularly Varying Increments}
\runtitle{State-independent IS for Heavy-tailed Random Walks}

\author{\fnms{Karthyek R. }\snm{A. Murthy,}\corref{}\thanksref{m1}\ead[label=e1]{kamurthy@tifr.res.in}}
\author{\fnms{Sandeep }\snm{Juneja}\thanksref{m1}\ead[label=e2]{juneja@tifr.res.in}}
\and
\author{\fnms{\newline Jose }\snm{Blanchet}\thanksref{m2}\ead[label=e3]{jose.blanchet@columbia.edu}}
\affiliation{Tata Institute of Fundamental Research\thanksmark{m1} and
  Columbia University\thanksmark{m2}}

\runauthor{Murthy, Juneja and Blanchet}

\address{School of Technology and Computer Science\\
Tata Institute of Fundamental Research\\
Mumbai, India\\
\printead{e1}\\
\phantom{E-mail:\ }\printead*{e2}}

\address{Department of Industrial Engineering\\ \quad  and Operations Research\\
  Columbia University\\ \quad In the city of New York\\
  \printead{e3}}

\begin{abstract}
  We develop importance sampling based efficient simulation techniques
  for three commonly encountered rare event probabilities associated
  with random walks having i.i.d. regularly varying increments;
  namely, 1) the large deviation probabilities, 2) the level crossing
  probabilities, and 3) the level crossing probabilities within a
  regenerative cycle. Exponential twisting based state-independent
  methods, which are effective in efficiently estimating these
  probabilities for light-tailed increments are not applicable when
  the increments are heavy-tailed.  To address the latter case, more
  complex and elegant state-dependent efficient simulation algorithms
  have been developed in the literature over the last few years.  We
  propose that by suitably decomposing these rare event probabilities
  into a dominant and further residual components, simpler
  state-independent importance sampling algorithms can be devised for
  each component resulting in composite unbiased estimators with
  desirable efficiency properties.  When the increments have infinite
  variance, there is an added complexity in estimating the level
  crossing probabilities as even the well known zero-variance measures
  have an infinite expected termination time. We adapt our algorithms
  so that this expectation is finite while the estimators remain
  strongly efficient.  Numerically, the proposed estimators perform at
  least as well, and sometimes substantially better than the existing
  state-dependent estimators in the literature.
\end{abstract}



\end{frontmatter}

\section{Introduction}
\label{SEC-INTRO}
In this paper, we develop importance sampling algorithms involving
simple, state-independent changes of measure for the efficient
estimation of large deviations, and level crossing probabilities of
random walks with regularly varying increments. Specifically, let
$X,X_1,X_2,\ldots$ denote a collection of independent and identically
distributed (i.i.d.)  random variables such that $\Pr\{X > x\} =
L(x)x^{-\alpha}$, for some $\alpha>1$ and a slowly varying
function\footnote{That is, $\lim_{x \rightarrow \infty} {L(tx)}/{L(x)}
  =1$ for any $t>0;$ see Section \ref{REG-VAR} for examples and more
  details} $L(\cdot)$. Note that $\alpha >2$ ensures finite variance
for $X$ whereas $\alpha <2$ implies that $X$ has infinite variance.
Set \[S_0 = 0, \quad S_n = X_1+\ldots+X_n, \text{ and } M_n = \max_{k
  \leq n} S_k, \text{ for } n \geq 1.\] Further, let $M := \sup_n
S_n,$ $\tau_b := \inf\{n \geq 1 : S_n > b\}$ and the regenerative
cycle duration $\tau := \inf\{ n \geq 1: S_n \leq 0\}.$ We are
interested in importance sampling based efficient estimation of:
 \begin{itemize}
 \item[1)] Large deviations probabilities $\Pr\{S_n - n \E X > b\}$
   for $b > \thr,$ with $\beta := (\alpha \wedge 2)^{-1}$ and any
   fixed $\epsilon > 0,$ as $n \nearrow \infty,$
 \item[2)] Level crossing probabilities $\Pr\{ \tau_b < \infty \},$ or
   equivalently, the tail probabilities $\Pr\{ M > b\},$ when $\E X <
   0$ and $b \nearrow \infty,$ 
 \item[3)] Level crossing probabilities within the regenerative cycle
   $\Pr\{\tau_b < \tau\},$ or equivalently, the tail probabilities of
   regenerative cycle maximum $\Pr \{ M_\tau > b\},$ when $\E X < 0,
   \alpha > 2$ and $b \nearrow \infty.$
 \end{itemize}
 Our methodology for estimating the large deviations probabilities
 easily extends to the efficient estimation of $\Pr\{S_N>u\}$ for
 random $N,$ when $N$ is light-tailed\footnote{As is well-known, $X$
   is light-tailed if the moment generating function $\E\left[
     \exp(\theta X)\right]$ is finite for some $\theta > 0,$ and is
   heavy-tailed otherwise.} and independent of increments $\{X_n\}$
 (popular in literature are $N$ fixed or geometrically distributed) as
 $u \nearrow \infty$ . However, in the interest of space, we do not
 explicitly consider the `random sum tail probabilities' estimation
 problem in this paper.

 Importance sampling via appropriate change of measure has been
 extremely successful in efficiently simulating rare events, and has
 been studied extensively in both the light and heavy tailed settings
 (see, e.g., \cite{MR2331321} for an introduction to rare event
 simulation and applications).  In importance sampling for random
 walks, state-dependence essentially means that the sampling
 distribution for generating the increment $X_k$ depends on the
 realized values of $X_1$, \ldots, $X_{k-1}$ (typically, through
 $S_{k-1}$); state-independence on the other hand implies that samples
 of $X_1, \ldots, X_n$ can be drawn independently. State-independent
 methods often enjoy advantages over state-dependent ones in terms of
 complexity of generating samples and ease of implementation.  The
 zero-variance changes of measure for estimating the large deviations
 and the level crossing probabilities are well known and are
 state-dependent (see, e.g., \cite{juneja2006rare}). While typically
 unimplementable, they provide guidance in search for implementable
 approximately zero-variance importance sampling techniques.

 In the light-tailed settings, large deviations analysis can be used
 to show that exponential twisting based state-independent importance
 sampling well approximates the zero-variance measures (see, e.g.,
 \cite{MR2331321}) for efficiently estimating the large deviations as
 well as the level crossing probabilities (see, e.g., \cite{MR1053850}
 and \cite{MR0418369}).  However, development of state-independent
 techniques for these probabilities is considered harder in the
 heavy-tailed settings. In \cite{ABH00}, Asmussen et al. provide an
 account of failure of simple large deviations based simulation
 methods that approximate zero-variance measure in heavy-tailed
 systems. Further, Bassamboo et al. \cite{MR2311409} prove that any
 importance sampling change of measure that prescribes increments to
 be drawn in an i.i.d. fashion cannot efficiently estimate
 probabilities of level crossing within a regenerative cycle of a
 heavy-tailed random walk.  The fact that the zero-variance measures
 for estimating both the large deviations and the level crossing
 probabilities are state-dependent, and the above mentioned negative
 results, have motivated research over the last few years in
 development of complex and elegant state-dependent algorithms to
 efficiently estimate these probabilities (see, e.g.,
 \citep{Dupuis:2007:ISS:1243991.1243995, MR2434174, MR2488534,
   Blanchet20122994, chan2012}).

 In this paper we introduce simple state-independent change of
 measures to estimate the large deviations and the level crossing
 probabilities with regularly varying increments.  We show that the
 proposed methods are provably efficient\footnote{We show that the
   estimators have asymptotically vanishing relative error; this
   corresponds to their coefficient of variation converging to zero as
   the event becomes rarer. We also have a related weaker notion of
   strong efficiency where the coefficient of variation of the
   estimators, and subsequently the number of i.i.d. replications
   required, remains bounded as the event becomes rarer.  Weak
   efficiency is another standard notion of performance in rare event
   simulation corresponding to a slow increase in the number of
   replications required as the event becomes rarer. These are briefly
   reviewed in Section \ref{sec:STR-EFF}.}  and perform at least as
 well as the existing state-dependent algorithms. Thus our key
 contribution is to question the prevailing view that one needs to
 resort to state-dependent methods for efficient computation of rare
 event probabilities involving `large number' of heavy-tailed random
 variables.  A key idea to be exploited in the estimation of
 probabilities considered is the fact that the corresponding rare
 event occurrence is governed by the ``single big jump'' principle,
 that is, the most likely paths leading to the occurrence of the rare
 event have one of the increments taking large value (see, for e.g.,
 \cite{MR2810144} and the references therein). Our approach for
 estimating the large deviations probability $\Pr\{S_n > b\}$ relies
 on decomposing it into a dominant and a residual component, and
 developing efficient estimation techniques for both. For estimating
 the level crossing probability $\Pr\{ \tau_b < \infty \}$, in
 addition to such a decomposition, we partition the event of interest
 into several blocks that are sampled using appropriate randomization.
 When the increments $X_n$ have infinite variance, there is an added
 complexity in estimating the level crossing probabilities $\Pr \{
 \tau_b < \infty\}$ as even the well known zero-variance measure is
 known to have an infinite expected termination time. We modify our
 algorithms so that this expectation remains finite while the
 estimators remain strongly efficient although they may no longer have
 asymptotically vanishing relative error.

Our specific contributions are as follows:

 \begin{enumerate}
 \item We provide importance sampling estimators that achieve
   asymptotically vanishing relative error for the estimation of
   $\Pr\{ S_n > b\},$ as $n \nearrow \infty.$ Given $n$ and $\epsilon
   > 0,$ our simulation methodology is uniformly efficient for values
   of $b$ larger than $n^{\frac{1}{2} + \epsilon}$ when the increments
   $X_n$ have finite variance, and for $b > n^{\frac{1}{\alpha} +
     \epsilon}$ in the case of increments having infinite variance --
   thus operating throughout the large deviations regime where the
   well-known asymptotics $\Pr \{ S_n > b\} \sim n\bar{F}(b)$
   hold. Further, this is the first instance that we are aware of
   where efficient simulation techniques for the large deviations
   probability include the case of increments having infinite
   variance, which is not uncommon in practical applications involving
   heavy-tailed random variables.

 \item For $\alpha >1$, we develop unbiased estimators for level
   crossing probabilities $\Pr\{\tau_b < \infty\}$ that achieve
   vanishing relative error as $b \nearrow \infty.$ These estimators
   require an overall computational effort that scales as $O(b)$ when
   the variance of increments $X_n$ is finite.  This is similar to the
   complexity of the zero variance operator since, as is well known,
   the latter requires order $\E [\tau_b | \tau_b < \infty]$
   computation in generating a single sample and this is known to be
   linear in $b$ when the variance of increments is finite (see
   \cite{Asmussen1996103}).  However, since $\E [\tau_b | \tau_b <
   \infty] =\infty$ for the case of increments having infinite
   variance, the zero-variance change of measure might not necessarily
   be a good benchmark, because from a computational standpoint any
   useful estimator needs to have finite expected termination time.
   For random walks with increments having infinite variance, we
   develop algorithms such that:
   \begin{enumerate}[(a)]
   \item When $\alpha > 1.5,$ the associated estimators are strongly
     efficient and have $O(b)$ expected termination time. As a
     converse, we also prove that for $\alpha < 1.5$ no algorithm can
     be devised in our framework that has both the variance and
     expected termination time simultaneously finite. The situation is
     more nuanced when $\alpha=1.5$ and depends on the form of the
     slowly varying function $L(\cdot)$.

   \item When $\alpha \leq 1.5,$ each replication of the estimator
     terminates in $O(b)$ time in expectation; also we require only
     $O(1)$ replications to achieve a given relative error, thus
     resulting in overall complexity of $O(b).$ 
   \end{enumerate}
   The above results for infinite increment variance, and in
   particular the bottleneck arising at $\alpha = 1.5$, closely mirror
   the results proved in \cite{Blanchet20122994} where vastly
   different state-dependent algorithms are considered.
 \item Similarly, for the level crossing probabilities $\Pr \{ \tau_b
   < \tau\},$ we partition the event into dominant and residual
   components, and devise changes of measure separately for the
   component events. The resulting importance sampling estimators are
   proved to be strongly efficient, as $b \nearrow \infty.$ This is
   significant considering the negative result of
   \cite{DBLP:journals/ior/BassambooJZ08} in context, where it is
   proved that no state-independent change of measure can be devised
   to efficiently simulate $\{ \tau_b < \tau \}.$ Our analysis thus
   informs that decomposing the event of interest in a suitable manner
   may be a reasonable way to address problems where designing
   importance sampling measures are known to be difficult.
\end{enumerate}

A brief discussion on practical applications and a literature review
may be in order: Efficient estimation of the level crossing
probability is important in many practical contexts, e.g., in
computing steady state probability of delays in $GI/GI/1$ queues and
in ruin probabilities in insurance settings (see, e.g.,
\cite{MR2331321}). Siegmund \cite{MR0418369} provides the first weakly
efficient importance sampling algorithm for estimating the level
crossing probabilities when the increments $X_n$ are light-tailed
using large deviations based exponentially twisted change of measure.
In \cite{MR1053850}, Sadowsky and Bucklew develop a weakly efficient
algorithm for estimating $\Pr( S_n > na) \text{ for } a > \E X,$ and
$X$ light-tailed, again using exponential twisting based importance
sampling distribution (also see \citep{MR1398051, MR2100018,
  MR2743897, MR2144566, agarwal2013} for related analysis).  This
problem is important mainly because it forms a building block to many
more complex rare event problems involving combination of renewal
processes: for examples in queueing, see \cite{MR970932} and in
financial credit risk modeling, see
\cite{DBLP:journals/mansci/GlassermanL05} and
\cite{DBLP:journals/ior/BassambooJZ08}.  

Research on efficient simulation of rare events involving heavy-tailed
variables first focussed on probabilities such as $\Pr \{S_N > b\}$ in
the simpler asymptotic regime where $N$ is fixed or geometrically
distributed and $b \nearrow \infty$.  In this simpler setting
state-independent algorithms are easily designed (see, e.g.,
\cite{ABH00,Juneja:2002:SHT:566392.566394,AK06}).  In \cite{6376313},
it is shown that a variant capped exponential twisting based
state-independent importance sampling, which does not involve any
decomposition, provides a strongly efficient estimator for the large
deviations probability that we consider in this paper. 

Statistical analysis reveals that heavy-tailed distributions are very
common in practice: in particular, heavy-tailed increments with
infinite variance are a convenient means to explain the long-range
dependence observed in tele-traffic data, and to model highly variable
claim sizes in insurance settings. Popular references to this strand
of literature include \cite{EKM97, RES97, MR1652283}.

The organization of the remaining paper is as follows: In Section
\ref{SEC-PRELIMS} we discuss preliminary concepts relevant to the
problems addressed. We propose our importance sampling method for
estimating the large deviations probability and prove its efficiency
in Section \ref{SEC-LD}.  In Section \ref{SEC-SIM-METH}, we develop
algorithms for estimating the level crossing probabilities $\Pr \{
\tau_b < \infty\}$.  Proofs of some of the key results pertaining to
efficiency and expected termination time of algorithms proposed in
Section \ref{SEC-SIM-METH} are presented in Section
\ref{PROOFS-KEY-RES}. The efficient simulation of level crossing
within a regenerative cycle is considered in Section
\ref{SEC-BCYC-SIM}.  Numerical experiments supporting our algorithms
are given in Section \ref{SEC-NUM-EG} followed by a brief conclusion
in Section \ref{SEC-CONC}.  Some of the more technical proofs are
presented in the appendix.

\section{Preliminary Background}
\label{SEC-PRELIMS}
In this section we briefly review the use of importance sampling in
estimating rare event probabilities.  We use Landau's notation for
describing asymptotic behaviour of functions: for given functions
$f:\mathbb{R}^+ \rightarrow \mathbb{R}^+$ and $g:\mathbb{R}^+
\rightarrow \mathbb{R}^+,$ we say $f(x) = O(g(x))$ if there exists
$c_1 > 0$ and $x_1$ large enough such that $f(x) \leq c_1g(x)$ for all
$x > x_1$; and $f(x) = \Omega(g(x))$ if there exists $c_2 > 0$ and
$x_2$ large enough such that $f(x) \geq c_2g(x)$ for all $x > x_2.$ We
use $f(x) = o(g(x))$ if $f(x)/g(x) \rightarrow 0,$ and $f(x) \sim
g(x)$ if $f(x)/g(x) \rightarrow 1,$ as $x \nearrow \infty.$ Throughout
this paper, if a probability measure is specified with a suffix, the
expectation and variance operators evaluated with respect to that
measure are specified with the same suffix. For example, $\E_n[\cdot]$
and $\text{Var}_n[\cdot]$ denote expectation and variance operators
associated with the measure $\Pr_n(\cdot).$

\subsection{Rare event simulation and importance sampling}
\label{sec:IS}
Let $A$ denote a rare event on the probability space $( \Omega,
\mathcal{F}, \Pr )$, i.e., $z:= \Pr(A) >0$ is small (in our setup $A$
corresponds to the events $\{S_n>b\}$ or $\{ \tau_b < \infty \}$).
Suppose that we are interested in obtaining an estimator $\hat{z}$ for
$z$ such that the relative error $|\hat{z} - z|/z $ is not more than
$\epsilon,$ with probability at least $1 - \delta,$ for given
$\epsilon$ and $\delta > 0.$ Naive simulation for estimating $z$
involves drawing $N$ independent samples of the indicator
$\mathbb{I}(A)$ and taking their sample mean as the estimator. For a
different measure $\Pt(\cdot)$ such that the Radon-Nikodym derivative
$d \Pr/d \Pt$ is well defined on $ A,$ we have:
\begin{equation*}
  \Pr(A) = \int _{ A } \frac{d \Pr}{d \Pt} (\omega
  ) d \Pt (\omega ) = \Et \left[ L \mathbb{I}_{A} \right],
\end{equation*}
where $L := d \Pr/d \Pt$ and $\Et[\cdot]$ is the expectation
associated with $\Pt(\cdot).$ Define $Z := L\mathbb{I}(A);$ then $Z$
is an unbiased estimator of $z$ under measure $\Pt(\cdot).$ If $N$
i.i.d samples $Z_1,\ldots,Z_N$ of $Z$ are drawn from $\Pt(\cdot),$
then by the strong law of large numbers we have:
\begin{equation*}
  \hat{z}_{_N} := \frac{Z_1 + \ldots + Z_N}{N}
  \rightarrow z \text{ a.s.},
\end{equation*} as $N \nearrow \infty.$
This method of arriving at an estimator is called \textit{importance
  sampling} (IS). The measure $\Pt(\cdot)$ is called the importance sampling
measure and $Z$ is called an importance sampling estimator. Using
Chebyshev's inequality allows us to find an upper bound on the number
of replications $N$ required to achieve the desired relative
precision:
\begin{align*}
  \Pr \left( \frac{|\hat{z}_{_N} - z|}{z} > \epsilon \right) \leq
  \frac{ \text{Var}_1 [\hat{z}_{_N}]}{z^2 {\epsilon }^2} =
  \frac{CV^2(Z)}{N \epsilon ^2}.
\end{align*}
Here $\text{Var}_1[\cdot] := \E_1[(\cdot)^2] - \E_1[\cdot]^2$ is the
variance operator associated with measure $\Pt(\cdot)$ and $CV(Z) =
\sqrt {\text{Var}_1[Z]}/z$ is the coefficient of variation of $Z.$
This enables us to conclude that if we generate at least
\begin{equation}
  N = \frac{CV^2(Z)}{\delta \epsilon ^2}
\label{NSAMP}
\end{equation}
i.i.d. samples of $Z$ for computing $\hat{z}_{_N}$, we can guarantee
the desired relative precision. In naive simulation we use the measure
$\Pr(\cdot)$ itself and have $Z = \mathbb{I}(A)$ as the estimator; so
the number of samples required in (\ref{NSAMP}) grows (roughly
proportional to $z^{-1}$) to infinity if $z \searrow 0.$ As is well
known, the choice $\Pr^*(\cdot) := \Pr (\cdot | A)$ as an importance
sampling measure yields zero variance for the associated estimator $Z
= z\mathbb{I}(A)$ (see e.g., \cite{MR2331321}); then every sample
obtained in simulation equals $z$ with $\Pr^*(\cdot)$ probability
1. However, the explicit dependence of $Z$ on $z,$ the quantity which
we want to estimate, makes this method impractical.

\subsection{Efficiency notions of algorithms}
\label{sec:STR-EFF} Consider a family of events $\{A_n: n \geq 1\}$
such that $z_n := \Pr (A_n) \searrow 0$ as the rarity parameter $n
\nearrow \infty.$ For an importance sampling algorithm to compute
$(z_n: n \geq 1),$ we come up with a sequence of changes of measure
$(\Pr_n(\cdot): n \geq 1)$ and estimators $( Z_n: n \geq 1 )$ such
that $\E_n Z_n = z_n.$ 
\begin{definition}
  The sequence $(Z_n: n \geq 1)$ of unbiased importance sampling
  estimators of $\{z_n: n \geq 1\}$, is said to achieve
  \textit{asymptotically vanishing relative error} if,
\begin{equation}
  \varlimsup _{ n \rightarrow \infty} \frac{\textnormal{Var}_n
    \left[ Z_n \right]}{z_n^2} = 0.
\end{equation} The sequence $(Z_n: n \geq 1)$ is said to be
\textit{strongly efficient} if,
\begin{equation}
  \sup_{n} \frac{\textnormal{Var}_n \left[
      Z_n \right]}{z_n^{2}} < \infty,
\end{equation}
and \textit{weakly efficient} if for all $\epsilon > 0,$
\begin{equation}
  \varlimsup_{ n \rightarrow \infty} \frac{\textnormal{Var}_n \left[
      Z_n  \right]}{z_n^{2-\epsilon}} = 0.
\end{equation}
\end{definition}
The significance of these definitions can be seen from \eqref{NSAMP}:
if an algorithm is strongly efficient, the number of simulation runs
required to guarantee the desired relative precision stays bounded as
$n \nearrow \infty$. If $\text{Var}(Z_n) = o\left(z^2_n\right),$ then
$(Z_n: n \geq 1)$ satisfies asymptotically vanishing relative error
property. As a result, it is enough to generate
$o(\delta^{-1}\epsilon^{-2})$ i.i.d. replications of the estimator. As
is apparent from the definition, all strongly efficient algorithms are
weakly efficient, and vanishing relative error is the strongest notion
among all three. Also it can be verified that naive
simulation is not even weakly efficient.

\subsection{Regularly varying tails}
\label{REG-VAR}
A function $L:\mathbb{R}^+ \rightarrow \mathbb{R}^+$ is said to be
\textit{slowly varying} at infinity if
\[\lim _{x \rightarrow \infty} \frac{L(tx)}{L(x)} =
1, \text{ for all } t > 0.\] Some examples of slowly varying functions
include $|\log x|^\beta \text{ for any } \beta \in \mathbb{R}, 1 -
e^{-x},$ etc. A random variable $X$ is said to be \textit{regularly
  varying} with index $-\alpha$ if for each $t > 0,$
\begin{equation*}
\lim_{x \rightarrow \infty} \frac{\Pr\{ X > tx \}}{ \Pr \{ X > x\}}
= t^{-\alpha}.
\label{RV-RV}
\end{equation*}
In other words, $\Pr\{ X > x\} = x^{-\alpha}L(x)$ for some slowly
varying function $L(\cdot).$ It can be easily verified that any
regularly varying random variable $X$ is heavy-tailed: that is,
$\E[\exp(\theta X)] = \infty$ for any $\theta > 0.$ These regularly
varying distribution functions capture the concept of polynomially
decaying tails, and form an important class of heavy-tailed
distributions. The following properties of regularly varying functions
will be useful in our analysis:\\\\
1) \textit{Karamata's theorem}: For any regularly varying function
$V(\cdot)$ with index $-\alpha,$ if $\beta$ is such that $\alpha -
\beta > 1,$ then
\begin{equation}
  \int_x^\infty u^\beta V(u)du \sim
  \frac{x^{\beta+1}V(x)}{\alpha-\beta-1}, \text{ as } x \nearrow
  \infty.
  \label{KARAMATA}
\end{equation}
This result, a part of Karamata's theorem (cf. Theorem 1 in Chapter
VIII.9 of \cite{feller1971introduction}), provides an asymptotic
characterization of integrated tails.\\\\
2) \textit{Potter's bounds}: Potter's bounds: If $L(\cdot)$ is a
slowly varying function, then as in Theorem 1.1.4 of \cite{MR2424161},
for any $\delta > 0,$ there exists a $t_\delta > 0$ such that for all
$t$ and $v$ satisfying $t \geq t_\delta$ and $vt \geq t_\delta,$
    \begin{equation}
      (1-\delta) \min \{v^\delta, v^{-\delta}\} \leq \frac{L(vt)}{L(t)}
      \leq (1+\delta)\max \{v^\delta, v^{-\delta}\}
      \label{LONG-TAIL-EXT}
    \end{equation}

\section{Simulation of $\{ \evA \}$}
\label{SEC-LD}
Let $X$ be a zero mean random variable with distribution $F(\cdot)$
satisfying the following:
\begin{assumption}
  The tail probabilities of $X$ are given by $\bar{F}(x) := \Pr \{ X >
  x \} = x^{-\alpha}L(x),$ for some slowly varying function $L(\cdot)$
  and $\alpha > 1.$ Further if $\text{Var}[X] = \infty,$ the tail
  probabilities of $X$ satisfy the following condition:
  \[\varlimsup_{x \rightarrow \infty} \frac{\Pr \{X < -x\}}{\Pr \{ X >
    x\}} < \infty.\]
\label{LEFT-TAIL-LIGHTER-ASSUMP} 
\end{assumption}
\noindent For the independent collection $(X_n: n \geq 1)$ of random
variables which are distributed identically as $X,$ define the random
walk $(S_n: n \geq 0)$ as below:
\[ S_0 = 0, \text{ and } S_n = X_1 + \ldots + X_n \text{ for } n \geq
1.\] In this section we devise a simulation procedure for estimating
the large deviation probabilities $\Pr\{ \evA \}$ 
and prove its efficiency as $n \nearrow \infty.$
For accomplishing this, we quickly review the following well-known
asymptotics: Let $\beta := (\alpha \wedge 2)^{-1}.$ When $\E X^2 <
\infty,$ we have that
\begin{align*}
  \frac{n\log n}{b^2} \int_{|x| \leq b} x^2F(dx) &\leq \frac{n \log
    n}{b^2} \E X^2 \searrow 0,\text{ as } n \nearrow \infty
\end{align*}
uniformly for $b > n^{\frac{1}{2} + \epsilon}.$ Similarly when $\E X^2
= \infty,$ since $\E X = 0,$ it follows from Assumption
\ref{LEFT-TAIL-LIGHTER-ASSUMP} that
\begin{align*}
  \frac{n}{b} \int_{|x| \leq b} xF(dx) &= - \frac{n}{b} \int_{|x|
    >  b} xF(dx)\\
  &= \frac{n}{b} \left( \int_{x < -b} |x|F(dx) + \int_{x > b}
    |x|\bar{F}(dx)\right)\\
  &\leq \frac{2n}{b} \left(\int_{x > b} x\bar{F}(dx)\right)(1+o(1))\\
  &= \frac{2n}{b} \left(b\bar{F}(b) + \int_b^\infty
    \bar{F}(u)du\right) (1+o(1))\\
  &= \frac{2n}{b} \left( b\bar{F}(b) + \frac{b\bar{F}(b)}{\alpha - 1}
  \right) (1+o(1)) \searrow 0, \text{ as } n \nearrow \infty,
\end{align*}
uniformly for $b > n^{\beta + \epsilon}.$ Then it follows from Theorem
3.3 of \cite{cline1991large} that for any $\epsilon > 0,$
\begin{equation}
  \sup_{ b > n^{\beta + \epsilon}} \left| \frac{\Pr \{ \evA \}}{\pn} -
    1 \right| = o(1), \text{ as } n \nearrow \infty.
  \label{LD-PROBS}
\end{equation}
A simple application of Bonferroni inequalities will yield 
\begin{equation}
  \Pr \{\max\{X_1,\ldots,X_n\} > b\} = n\bar{F}(b)\left( 1 -
    \frac{1}{2}\bar{F}(b) + \frac{\theta}{6}\left(
      n\bar{F}(b)\right)^2\right), 
  \label{MAX-BONF-INEQ}
\end{equation}
for some $\theta$ in $(0,1).$ This indicates that the tail asymptotics
of maximum and the sum of increments $\{ X_1,\ldots,X_n\}$ match
asymptotically.


The strategy for simulation is to partition the event $\{ S_n > b \}$
based on whether the maximum of the increments $\{ X_1, \ldots, X_n\}$
has exceeded the large value $b$ or not:
\begin{align*}
 \Ad(n,b) &:= \left\{ \evA, \max_{k \leq n} X_k \geq b\right\} \text{
  and }\\
\Ar(n,b) &:= \left\{ \evA, \max_{k \leq n} X_k < b\right\}.
\end{align*}
Such a partition is considered in \cite{DBLP:journals/questa/Juneja07}
for the simulation of $\{ S_n>b \}$ when $n$ is fixed. We prove the
following result in the appendix:
\begin{proposition}
  Under Assumption \ref{LEFT-TAIL-LIGHTER-ASSUMP}, given any
  $\epsilon > 0,$
  \[ \sup_{b > n^{\beta + \epsilon}} \frac{\Pr
    \left(\Ar(n,b)\right)}{n \bar{F}(b)} = o(1) \text{ as } n \nearrow
  \infty.\]
  \label{LD-RES-PROB}
\end{proposition}
Therefore, the probability of the event $\Ar$ is vanishingly small
compared to the probability of $\Ad$ as $n \nearrow \infty;$ the
suffixes stand to indicate that $\Ad$ is the dominant way of
occurrence of $\{\evA\}$ for large $n,$ and $\Ar$ has only residual
contributions. We estimate $\Pr(\Ad)$ and $\Pr(\Ar)$ independently via
different changes of measure that typify the way in which the
respective events occur, and add the individual estimates to arrive at
the final estimator for $\Pr\{\evA\}.$

\subsection{Simulating $\Ad$}
\label{SUBSEC-SIM-AD}
For the simulation of $\Ad,$ we follow the two-step procedure outlined
in \cite{chan2012}:
\begin{enumerate}
\item Choose an index $I$ uniformly at random from $\{1,\ldots,n\}$
\item For $k=1,\ldots,n,$ generate a realization of $X_k$ from
  $F(\cdot|X_k \geq b)$ if $k=I;$ otherwise, generate $X_k$ from
  $F(\cdot)$
\end{enumerate}
Let $\Ph(\cdot)$ denote the measure induced when the increments are
generated according to the above procedure; for brevity, we have
chosen not to highlight the dependence of the importance sampling
change of measure $\Ph(\cdot)$ on $n$ and $b$ in the
notation. 
Note that the probability measure $\Pr(\cdot)$ is
absolutely continuous with respect to $\Ph(\cdot)$ when restricted to
$\Ad.$ We have,
\begin{align*}
  d\Ph\left( x_1,\ldots,x_n\right) &= \sum_{k=1}^n
  \frac{1}{n}\cdot\frac{dF(x_1)\ldots
    dF(x_n)}{\bar{F}(b)}\mathbf{1}(x_k \geq b).
\end{align*}
Therefore the likelihood ratio on the set $\Ad$ is given by,
\[ \frac{d\Pr}{d\Ph}(X_1,\ldots,X_n) = \frac{n\bar{F}(b)}{\#\{X_i \geq
  b: 1 \leq i \leq n\}},\] and the resulting unbiased estimator for
the evaluation of $\Pr(\Ad)$ is,
\begin{equation}
  \Zd(n,b) := \frac{n\bar{F}(b)}{\#\{X_i \geq b: 1 \leq i \leq
    n\}}\mathbb{I}(\Ad).
  \label{EST-DOM}
\end{equation}
Generate $N$ independent realizations of $\Zd$ and take their sample
mean as an estimator of $\Pr(\Ad).$ To evaluate how large $N$ should
be chosen so that the computed estimate satisfies the given relative
error specification, we need to obtain bounds on the variance of
$\Zd.$ Since $\#\{X_i \geq b: 1 \leq i \leq n\}$ is at least 1, when
the increments are drawn following the measure $\Ph(\cdot),$ we have:
$\Zd(n,b) \leq \pn,$ and hence,
\[ \Eh \left[ \Zd^2(n,b) \right] \leq \left( n\bar{F}(b) \right)^2.\]
Also $\Eh \left[ \Zd(n,b) \right] = \Pr(\Ad(n,b)) \sim \Pr\{\evA\}
\sim \pn,$ as $n \nearrow \infty.$ Therefore we get,
\begin{equation}
  \text{Var}_1\left[\Zd(n,b)\right] = o\left( \left( n\bar{F}(b)
    \right)^2 \right), \text{ as } n \nearrow \infty.
  \label{VAR-DOM}
\end{equation}

\begin{remark}
  Since $\Pr\{S_n > b, M_n>b\} = n\Pr\{S_n > b, M_n>b,M_n=X_1\},$ one
  can estimate $\Pr\{S_n > b, M_n>b,M_n=X_1\}$ efficiently by
  simulating $X_1$ from $F(\cdot|X_1 > b)$ and the other increments
  from $F(\cdot).$ This avoids the simulation of an additional random
  variable $I.$ However, we have presented the two step procedure
  above so that the simulation procedures introduced later in the
  paper appear intuitive.
\label{WHY-GEN-I}
\end{remark}

\begin{remark}
  If the increments $X_1,\ldots,X_n$ are not identically distributed,
  and if at least one of the increments is regularly varying, then it
  can be verified that the following modification to the simulation of
  auxiliary random variable $I$ would suffice: Say $X_j \sim
  F_j(\cdot).$ Then choose $I=i$ from $\{1,\ldots,n\}$ with
  probability $\bar{F}_i(b)/\sum_{j=1}^n\bar{F}_j(b).$
  \label{NON-IDENT-INC}
\end{remark}

\subsection{Simulating $\Ar$}
\label{SUBSEC-SIM-AR}
We see that all the increments $\{ X_1,\ldots, X_n\}$ are bounded from
above by $b$ on the occurrence of event $\Ar.$ Though the bound on the
increments vary with $n,$ we can employ methods similar to exponential
twisting of light-tailed random walks to simulate the event $\Ar,$ as
illustrated in this section. For given $b,$ define
\begin{align*}
\Lambda_b(\theta) &:= \log
  \left(\int_{-\infty}^{b}\exp(\theta x)F(dx) \right), \quad
  \theta \geq 0.
\end{align*}
Since the upper limit of integration is $b,$ $\Lambda_b(\cdot)$ is
well-defined for any positive value of $\theta.$ For given values of
$n$ and $b,$ consider the distribution function $F_{\theta}(\cdot)$
satisfying,
\begin{align*}
  \frac{dF_{\theta}(x)}{dF(x)} = \exp (\tnb x -
  \Lambda_b(\tnb))\mathbf{1}(x < b),
\end{align*}
for all $x \in \mathbb{R}$ and some $\tnb > 0.$ Now the prescribed
procedure is to just obtain independent samples of the increments
$\{X_1,\ldots,X_n\}$ from $F_{\theta}(\cdot)$ and adjust via the
likelihood ratio resulting due to the procedure of sampling from a
different distribution $F_{\theta}(\cdot).$

Let $\Pth(\cdot)$ denote the measure induced by sampling increments
i.i.d from $F_{\theta}(\cdot).$ As before, for brevity, we have chosen
not to highlight the dependence on parameters $n$ and $b$ in the
notations $F_{\theta}(\cdot)$ and $\Pth(\cdot).$ 
For given values of $n$ and $b,$ we have the following unbiased
estimator for the computation of $\Pr(\Ar):$
\begin{align}
  \Zr(n,b) := \exp\left(-\tnb S_n + n\Lambda_b(\tnb)\right)
  \mathbb{I}(\Ar).
  \label{EST-RES}
\end{align}
Now generate independent replications of $\Zr$ and take their sample
mean as the computed estimate for $\Pr(\Ar).$ However it remains to
choose $\tnb.$ Since $S_n$ is larger than $b$ on $\Ar,$
\begin{align*}
  \Zr(n,b) \leq \exp\left(-\tnb b + n\Lambda_b(\tnb)\right)
  \mathbb{I}(\Ar).
\end{align*}
If we choose
\begin{align}
  \tnb &:= -\frac{\log\left(\pn\right)}{b}, \text{ then } \label{THETA-N}\\
  \quad \Zr(n,b) &\leq \pn \exp\left( n\Lambda_b(\tnb)\right)
  \mathbb{I}(\Ar). \label{RES-EST-BND}
\end{align}
We use Lemma \ref{LEM-LD}, which is proved in the appendix, to obtain
an upper bound on the second moment of the estimator $\Zr.$
\begin{lemma}
  Under Assumption \ref{LEFT-TAIL-LIGHTER-ASSUMP}, for the choice of
  $\tnb$ as in \eqref{THETA-N},
  \[ \exp \left( \Lambda_b(\tnb) \right) \leq
  1+\frac{1}{n}(1+o(1)), \] as $n \nearrow \infty,$ uniformly for $b >
  \thr.$
  \label{LEM-LD}
\end{lemma}
\noindent Therefore there exists a constant $c$ such that
\[ \exp\left( n\Lambda_b(\tnb)\right) \leq c,\] for all admissible
values of $n$ and $b.$ We evaluate the second moment of the estimator
$\Zr$ through the equivalent expectation operation corresponding to
the original measure $\Pr(\cdot)$ as below:
\begin{align*}
  \Eth \left[ \Zr^2(n,b) \right] = \E \left[ \Zr(n,b) \right] \leq c
  \pn \Pr(\Ar),
\end{align*}
where the last inequality follows from \eqref{RES-EST-BND} and Lemma
\ref{LEM-LD}. From Proposition \ref{LD-RES-PROB}, we have that $\Pr
(\Ar) = o \left( \pn \right).$ Therefore,
\begin{align}
  \text{Var}_2\left[\Zr(n,b)\right] = o\left( (\pn)^2\right), \text{ as
  } n \nearrow \infty,
\label{VAR-RES}
\end{align}
thus arriving at the following theorem:
\begin{theorem}
  If the realizations of the estimators $\Zd$ and $\Zr$ are generated
  respectively from the measures $\Pt(\cdot)$ and $\Pth(\cdot),$ and
  if we let
  \[ Z(n,b) := \Zd(n,b) + \Zr(n,b),\] then under Assumption
  \ref{LEFT-TAIL-LIGHTER-ASSUMP}, the family of estimators $(Z(n,b):n
  \geq 1, b > \thr)$ achieves asymptotically vanishing relative error
  for the estimation of $\Pr\{\evA\},$ as $n \nearrow \infty;$ that
  is,
  \[ \frac{\textnormal{Var}_{n,b}\left[Z(n,b)\right]}{\Pr\{\evA\}^2} =
  o(1),\] as $n \nearrow \infty,$ uniformly for $b > \thr.$
\label{EFF-LD}
\end{theorem}
\noindent Here $\textnormal{Var}_{n,b}[\cdot]$ denotes the variance
operator resulting due to the composite procedure of drawing
realizations of $\Zd$ and $\Zr$ from the measures $\Pt(\cdot)$ and
$\Pth(\cdot)$ respectively.
\begin{proof}
  Since the realizations of $\Zd$ and $\Zr$ are obtained independent
  of each other, the variance of $Z$ is just the sum of variances of
  $\Zd$ and $\Zr$ computed according to the measures from which they
  are generated; the proof is now evident from \eqref{VAR-DOM},
  \eqref{VAR-RES} and \eqref{LD-PROBS}.
\end{proof}
\begin{remark}
  A consequence of the above theorem is that, due to \eqref{NSAMP},
  the number of i.i.d. replications of $Z(n,b)$ required to achieve
  $\epsilon$-relative precision with probability at least $1-\delta$
  is at most $o(\epsilon^{-2}\delta^{-1}),$ which is independent of
  the rarity parameters $n$ and $b.$ In our algorithm each replication
  demands $O(n)$ computational effort, thus requiring an overall
  computational cost of $O(n),$ as $n \nearrow \infty.$
\end{remark}

\begin{remark}
  One can easily check that, this same simulation procedure can also
  be used to efficiently compute probabilities $\Pr\{S_N>b\}$ when $N$
  is a light-tailed random variable independent of the increments
  $X_n.$
\end{remark}

\section{Simulation Methodology for $\{\tau_b < \infty\}$}
\label{SEC-SIM-METH}

As before, the sequence $(S_n: n \geq 0)$ with $S_0:=0$ and $S_n :=
X_1 + \ldots + X_n$ represents the random walk associated with the
i.i.d collection $(X_n: n \geq 1).$ We have $\E X_n = 0,$ and $\Pr\{
X_n > x\} = x^{-\alpha}L(x)$ for some slowly varying function
$L(\cdot)$ and $\alpha > 1.$ Given $\mu > 0,$ let
\[M:=\sup_n \ (S_n-n\mu).\] Since $(S_n-n\mu: n \geq 0)$ is a random
walk with negative drift, the random variable $M$ is proper. For $b >
0,$ recall that the first-passage time $\tau_b$ is defined as $\tau_b
:= \inf\{n \geq 0: S_n-n\mu > b\}.$ In this section we present
simulation methods for the efficient computation of
\[\Pr\{ M > b\} = \Pr \{ \tau_b < \infty\}, \text{ as } b \nearrow
\infty.\]

Naive simulation of $\{ \tau_b < \infty\}$ will require generation of
all the increments until the partial sum $S_n-n\mu$ exceed $b.$ Due to
the negative drift of the random walk $(S_n - n\mu: n \geq 0),$ we
have $\tau_b \nearrow \infty$ a.s. as $b \nearrow \infty,$ and hence
this method is not computationally feasible. To counter the prospect
of generating uncontrollably large number of increment random
variables in simulation, we re-express $\Pr \{ \tau_b < \infty\}$ as
below: Consider a strictly increasing sequence of integers $( n_k : k
\geq 0)$ with $n_0 = 0;$ also fix $p := (p_k : k \geq 1)$ satisfying
$p_k > 0$ for all $k$ and $\sum_k p_k = 1;$ the vector $p$ can be seen
as a probability mass function on positive integers. Consider an
auxiliary random variable $K$ which takes the value of positive
integer $k$ with probability $p_k.$ Then
\begin{align}
  \Pr \{ \tau_b < \infty\} &= \sum_{k \geq 1} p_k \frac{\Pr \{
    \evk\}}{p_k}
  \nonumber\\
  &= \E \left[ \E \left[ \frac{\Pr \{ n_{_{K-1}} < \tau_b \leq
        n_{_K}\}}{\pK} \left| \frac{}{} \right. K \right]\right].
\label{RE-EXP}
\end{align}

Now in a simulation run, if the realized value of the auxiliary random
variable $K$ is $k,$ generate a sample from a probability measure,
possibly different from $\Pr(\cdot),$ of a random variable $Z_k$ that
has $\Pr\{\evk\}$ as its expectation under the changed measure. Then
equation \eqref{RE-EXP} assures that taking the sample mean of
i.i.d. replications of $Z_K/p_{_K}$ following the changes of measure
(to be explained in Section \ref{SUBSEC-SIM-LOC}) for the generation
of $\{Z_k: k \geq 1\}$ will yield an unbiased estimator for the
quantity $\Pr \{ \tau_b < \infty\}.$

The performance of any importance sampling algorithm following the
outlined procedure will depend crucially on the choice of
probabilities $p_k,$ and the changes of measure employed to estimate
$\Pr\{ \evk \},$ for $k \geq 1.$ The sequence $(n_k : k \geq 0)$
partitions non-negative integers into `blocks' $((\nkm, n_k]: k \geq
1).$ For reasons that will be clear later, we choose the blocks
$(\nkm, n_k]$ in the following manner: Fix a positive integer $r > 1$
and let,
\[n_0 = 0, \hspace{2pt} n_k = r^k, \text{ for } k \geq 1.\] In the
following section, we present related asymptotics that will be useful
in the efficiency analysis of the algorithms that are
developed.

\subsection{Related Asymptotics}
\label{ASYMP-LC-SEC}
Recall that 
\[\tau_b := \inf\{k: S_k > b + k \mu\} \text { and } M :=
\sup_n(S_n-n\mu).\] The events $\left\{ M > b \right\}$ and $\{ \tau_b
< \infty\}$ are the same. Let $\Fi(x) := \int_{x}^{\infty}
\bar{F}(u)du $ denote the integrated tail of $\bar{F}(\cdot).$ Under
Assumption \ref{LEFT-TAIL-LIGHTER-ASSUMP}, it is well known (see, for
example, \cite{Veraverbeke197727}) that,
\begin{equation}
  \Pr \{\tau_b < \infty\} \sim \frac{1}{\mu}\Fi(b), \text{ as } b
  \nearrow \infty.
  \label{ASYMP}
\end{equation}
The asymptotics \eqref{ASYMP} hold for level crossing probabilities of
random walks under more general increment distributions (see, for
example, \cite{Korshunov199797}).

The following finite-horizon asymptotics are also available if we make
this non-restrictive smoothness assumption on the tail probabilities
$\bar{F}(\cdot):$
\begin{assumption}
  There exists a $t_0 > 0$ such that the slowly varying function
  $L(\cdot)$ in $\bar{F}(x) = x^{-\alpha}L(x)$ is continuously
  differentiable for all $t \geq t_0.$ Further $L(\cdot)$ satisfies,
  \[ L'(x) = o\left( \frac{L(x)}{x}\right), \text{ as } x \nearrow
  \infty.\]
  \label{SMOOTHNESS-ASSUMP}
\end{assumption}
If $X$ is such that $\text{Var}[X] < \infty$ and it satisfies
Assumptions \ref{LEFT-TAIL-LIGHTER-ASSUMP} and
\ref{SMOOTHNESS-ASSUMP}, then from Theorem 6 of
\cite{doi:10.1137/S0040585X97978877}, we have uniformly in $n$ that,
\begin{equation}
  \Pr \{ \tau_b \leq n\} = \left( \sum_{j=1}^n \bar{F}(b+j\mu) \right) \left( 1 +
    O\left(\frac{1}{b}\right)  \right) + o\left( \sqrt{b \wedge n} \  
    \bar{F}(b)\right). 
  \label{FINITE-HOR-ASYMP-FIN-VAR}
\end{equation}
When $\text{Var}[X] = \infty,$ under Assumptions
\ref{LEFT-TAIL-LIGHTER-ASSUMP} and \ref{SMOOTHNESS-ASSUMP}, it follows
from Theorem 2.4 of \cite{borovkovboxma} that uniformly for all $n,b$
satisfying $n\bar{F}(b) = o(1),$
\begin{equation}
  \Pr \{ \tau_b \leq n \} = \left( \sum_{j=1}^n \bar{F}(b+j\mu)
  \right)  \left( 1 +
    O\left(\frac{n^{\frac{1}{\alpha}+\epsilon}}{b}\right) \right)
  \label{FINITE-HOR-ASYMP-INF-VAR}
\end{equation}
for every $\epsilon > 0.$

The following characterization of the zero-variance measure $\Pr\{
\cdot| \tau_b < \infty\}$ (see Theorem 1.1 of \cite{Asmussen1996103})
sheds light on how the first passage over a level $b$ happens
asymptotically: If we use $a(b):=\Fi(b)/\bar{F}(b),$ then conditional
on $\tau_b < \infty,$
\begin{equation}
  \left( \frac{\tau_b}{a(b)}, \left( \frac{S_{\lfloor u\tau_b
          \rfloor}}{\tau_b}: 0 \leq u < 1\right), \frac{S_{\tau_b} -
      b}{a(b)}\right)
  \Longrightarrow \left( \frac{Y_0}{\mu}, (-u\mu : 0 \leq u < 1),
    Y_1\right)
\label{COND-LT}
\end{equation}
in $\mathbb{R} \times D[0,1) \times \mathbb{R}.$ The joint law of
$Y_0,Y_1$ is defined as follows: for $y_0, y_1 \geq 0, \Pr \{Y_0 >
y_0, Y_1 > y_1\} = \Pr \{ Y_1 > y_0 + y_1\}$ with $Y_0 \overset{d}{=}
Y_1,$ and
\[ \Pr \{ Y_1 > y_1\} =
\frac{1}{\left(1+y_1/(\alpha-1)\right)^{\alpha-1}}.\]

\subsection{Efficient simulation of $\{ \evk \}$}
\label{SUBSEC-SIM-LOC}
In this section we identify importance sampling changes of measure for
the efficient computation of the probabilities $\Pr \{ \evk
\}.$ 
Define the following events:
\[ A_k = \bigcup_{i=\nkm + 1}^{n_k}\left\{ X_i > b+i\mu \right\}
\text{ and } B_k = \bigcap_{i=1}^{n_k} \left\{ X_i < b+ \nkm\mu
\right\}.\] The events $A_k$ and $B_k$ are defined in the same spirit
as that of $\Ad$ and $\Ar$ in the simulation of $\{\evA\}$ in Section
\ref{SEC-LD}: the event $A_k$ includes sample paths that have at least
one ``big'' jump of appropriate size in one of the increments indexed
between $\nkm$ and $n_k,$ whereas on the other set $B_k,$ we have all
the increments bounded from above. As in the simulation of large
deviation probabilities of sums of random variables in Section
\ref{SEC-LD}, we can partition the event $\{\evk\}$ into:
\[ \{ \evk, A_k\}, \{ \evk, B_k\} \text{ and } \{ \evk,
\bar{A}_k\cap\bar{B}_k\},\] and arrive at unbiased estimators for
their probabilities separately via different importance sampling
measures. Here $\bar{A}$ denotes complement of the set $A.$


\subsubsection{Simulating $\{ \evk, A_k\}$}
\label{SIM-AK}
We prescribe the following two step procedure: Let $q_k(b) :=
\sum_{i=\nkm+1}^{n_k} \bar{F}(b+i\mu).$
\begin{enumerate}
\item Choose an index $J \in \{\nkm+1, \ldots, n_k\}$ such that $\Pr
  \{J = n\} = \bar{F}(b+ n \mu)/q_k(b),$ for $\nkm < n \leq n_k.$
\item Simulate the increment $X_n$ from $F(\cdot|X_n \geq b + n \mu),$
  if $n = J;$ otherwise, simulate $X_n$ from $F(\cdot),$ for any $n
  \leq n_k.$
\end{enumerate}
In this sampling procedure, we induce the `big' jumps typically
responsible for the occurrence of $\{ \evk \}$ with suitable
probabilities by sampling from the conditional distribution
$F(\cdot|X_J \geq b + J\mu).$ This sampling procedure results in the
importance sampling measure $\Pr _{k,1}(\cdot)$ characterised by:
\begin{align*}
  d\Pr _{k,1} (x_1,\ldots,x_{n_k}) := \sum_{i=\nkm+1}^{n_k}
  \frac{\bar{F}(b+ i \mu)}{q_k(b)}. \frac{dF(x_1)\ldots
    dF(x_{n_k})}{\bar{F}(b+ i \mu)}\mathbf{1}(x_i \geq b+ia).
\end{align*}
This in turn yields a likelihood ratio,
\begin{align*}
  \frac{d\Pr}{d\Pr_{k,1}}(X_1,\ldots,X_{n_k}) = \frac{q_k(b)}{\#\{ X_i
    \geq b+i \mu : \nkm < i \leq n_k\}},
  \label{LOC-LLR-1}
\end{align*}
on the set $A_k.$ Then we have,
\begin{equation}
  Z_{k,1}(b) := \frac{q_k(b)}{\#\{X_i \geq b+i \mu : \nkm < i \leq
    n_k \}} \mathbb{I}(\evk, A_k)
  \label{LOC-EST-1}
\end{equation}
as the unbiased estimator for the quantity $\Pr\{ \evk, A_k\}.$ Here
note that $\mathbb{I}(\evk, A_k) = 1$ a.s. under $\Pr_{k,1}(\cdot).$

\subsubsection{Simulating $\{ \evk, B_k\}$}
\label{SIM-BK}
On the event $B_k,$ none of the random variables $X_1,\ldots,X_{n_k}$
exceed the level $(b+\nkm\mu);$ since these increments are bounded (on
$B_k)$, we can draw their samples from an appropriately truncated,
exponentially twisted variation of $F(\cdot),$ as in Section
\ref{SUBSEC-SIM-AR}, without losing absolute continuity on $\{ \evk,
B_k\}$. For estimating $\Pr \{ \evk, B_k\},$ we draw samples of
$X_1,\ldots, X_{\tau_b \wedge n_k}$ independently from the
distribution $F_{k}(\cdot)$ satisfying,
\begin{align}
  \frac{dF_k(x)}{dF(x)} &= \exp(\theta_k x -
  \Lambda_k(\theta_k))\mathbf{1}(x < b+ \nkm \mu), \quad x \in
  \mathbb{R}; \nonumber\\
  \text{here, } \theta_k (= \theta_k(b)) &:= \frac{-\log
    (n_k\bar{F}(b+ \nkm \mu))}{b+ \nkm \mu}, \text{ and } \label{THETA-K}\\
  \Lambda_k(\theta) &:= \log \left(\int_{-\infty}^{b+ \nkm
      \mu}\exp(\theta_k x)F(dx) \right), \quad \theta \geq
  0. \label{LAMBDA-K}
\end{align}
Let $\Pr_{k,2}(\cdot)$ be the measure induced by drawing samples as
above. Then the resulting likelihood ratio on $\{ \evk, B_k\}$ is:
\begin{align*}
  \frac{d\Pr}{d \Pr_{k,2}}(X_1,\ldots,X_{n_k}) =
  \exp\left(-\theta_kS_{\tau_b} + \tau_b\Lambda_k(\theta_k)\right).
\end{align*}
The associated estimator for computing $\Pr \{ \evk, B_k \}$ is:
\begin{align}
  Z_{k,2}(b) := \exp\left(-\theta_kS_{\tau_b} +
    \tau_b\Lambda_k(\theta_k)\right) \mathbb{I}(\evk,B_k)
  \label{LOC-EST-2}
\end{align}

\subsubsection{Simulating $\{ \evk, \bar{A}_k\cap\bar{B}_k\}$}
\label{SIM-ABK}
We draw samples in a two step procedure similar to that in the Section
\ref{SIM-AK}.
\begin{enumerate}
\item Choose an index $J$ uniformly at random from $\{1, \ldots,
  n_k\}$
\item Simulate the increment $X_n$ from $F(\cdot|X_n \geq b + \nkm
  \mu),$ if $n = J;$ otherwise, simulate $X_n$ from $F(\cdot),$ for
  any $n \leq n_k.$
\end{enumerate}
If $\Pr_{k,3}(\cdot)$ denotes the change of measure induced by drawing
samples according to the above procedure, then the likelihood ratio on
the set $\{ \evk, \bar{A}_k\cap\bar{B}_k\}$ is:
\begin{align*}
  \frac{d\Pr}{d\Pr_{k,3}}(X_1,\ldots,X_{n_k}) = \frac{n_k
    \bar{F}(b+\nkm \mu)}{\#\{ X_i \geq b + \nkm \mu : 1 < i \leq
    n_k\}}.
\end{align*}
The resulting estimator for the computation of $\Pr \{ \evk,
\bar{A}_k\cap\bar{B}_k \}$ is:
\begin{align}
  Z_{k,3}(b) := \frac{n_k \bar{F}(b+\nkm\mu)}{\#\{ X_i \geq b + \nkm
    \mu : 1 < i \leq n_k \}} \mathbb{I}\left(\evk,
    \bar{A}_k\cap\bar{B}_k \right).
  \label{LOC-EST-3}
\end{align}

As in Section \ref{SEC-LD}, the estimator for $\Pr \{\evk\}$ can be
obtained by summing the estimators of component events $\Pr\{ \evk,
A_k\}, \Pr\{ \evk, B_k\}, \text{ and } \Pr\{ \evk,
\bar{A}_k\cap\bar{B}_k\}:$
\begin{align}
  Z_k(b) := Z_{k,1}(b) + Z_{k,2}(b) + Z_{k,3}(b).
  \label{LOC-EST}
\end{align}

\subsection{Simulation of $\{\tau_b < \infty\}$ - the finite variance
  case}
\label{SEC-FIN-VAR}
Here we develop on the ideas stated at the beginning of Section
\ref{SEC-SIM-METH}. We have the increasing sequence of integers $(n_k:
k \geq 0),$
\[ n_0 = 0, n_k = r^k \text{ for } k \geq 1,\] for some integer $r >
1.$ Further, we have an auxiliary random variable $K$ taking values in
positive integers according to the probability mass function $(p_k: k
\geq 1).$ As in \eqref{RE-EXP}, we re-express the quantity of interest
as:
\begin{align*}
  \Pr \{ \tau_b < \infty\} &= \E \left[ \E \left[ \frac{\Pr \{ n_{_{K-1}} < \tau_b \leq
        n_{_K}\}}{\pK} \left| \frac{}{} \right. K \right]\right].
\end{align*}
From \eqref{LOC-EST}, we have estimators $\{ Z_k(b): k \geq 1\}$ that
can be used to compute the corresponding probabilities $\{ \Pr \{ \evk
\}: k \geq 1 \}.$ Consider the following simulation procedure:
\begin{enumerate}
\item Draw a sample of $K$ such that $\text{\Pr}\{K=k\} = p_k.$
\item Conditional on the realized value of $K,$
  \begin{enumerate}
  \item[2a)] Generate a realization of $Z_K(b)$ as in Section
    \ref{SUBSEC-SIM-LOC}.
  \item[2b)] Return ${Z_K(b)}/{\pK}.$
  \end{enumerate}
\end{enumerate}

We present the sample mean of the values returned by $N$ independent
simulation runs of the above procedure as our final estimate of $\Pr\{
\tau_b < \infty\}.$ Let $Q(\cdot)$ denote the probability measure in
the path space induced by the generation of increment random variables
as a result of this sampling procedure; let $\E^Q[\cdot]$ and
$\text{Var}^Q[\cdot]$ be the expectation and variance operator
associated with the measure $Q(\cdot).$ Given $b > 0,$ the overall
unbiased estimator for the computation of $\Pr \{\tau_b < \infty\}$
is,
\[ Z(b) := \frac{Z_K(b)}{\pK}.\]

Note that the number of independent simulation runs needed to achieve
a desired relative precision, as in \eqref{NSAMP}, is directly related
to the sampling variance of $Z(b).$ If $(Z(b):b > 0)$ offer
asymptotically vanishing relative error, we just need
$o(\epsilon^{-2}\delta^{-1})$ independent replications of the
estimator.  However, as pointed in \cite{MR0223065}, and further
justified in \cite{MR1180030}, both the variance of an estimator and
the expected computational effort required to generate a single sample
are important performance measures, and their product can be
considered as a `figure of merit' in comparing performance of
algorithms that provide unbiased estimators of $\Pr \{\tau_b <
\infty\}$.  For any given $b,$ let $\nu_b$ denote the largest index of
the increment random variables $(X_n:n \geq 1)$ considered for
simulation in a particular simulation run. The expectation of $\nu_b$,
then gives a measure of the expected number of increment random
variables generated, and subsequently of the expected computational
effort in every simulation run.  In particular, the latter may be
bounded from above by a constant $C >0$ times the expectation of
$\nu_b$.

In a single run of the above procedure, if the realized value of $K$
is $k,$ we look for estimating $\Pr\{ \evk \}$ which does not entail
the generation of more than $n_k$ increment random variables, thus
ensuring termination. In particular, $n_{_{K-1}} \leq \nu_b \leq
n_{_K}$.  The following theorems give a measure of both the variance
and the expected computational effort per replication of $Z(b)$ for a
specific choice of the probabilities $p_k$. Recall that $Q(\cdot)$ is
the probability measure that governs the law of $Z(b)$ when the random
variables $Z_{K,j}(b)$ are generated as explained in Sections
\ref{SIM-AK}, \ref{SIM-BK} and \ref{SIM-ABK}.

In all the theorems that follow it is assumed that the common
distribution $F(\cdot)$ of the increments satisfy Assumptions
\ref{LEFT-TAIL-LIGHTER-ASSUMP} and \ref{SMOOTHNESS-ASSUMP}.
\begin{theorem}
For
\begin{equation}
 p_k = \frac{\Fi(b+\nkm \mu) - \Fi(b+n_k \mu)}{\Fi(b)}, k \geq 1,
 \label{PK-SEL}
\end{equation}
the family of unbiased estimators $\left( Z(b): b > 0 \right)$
achieves asymptotically vanishing relative error for the computation
of $\Pr\{ \tau_b < \infty\},$ as $b \nearrow \infty;$ that is:
\[ \varlimsup_{b \rightarrow \infty} \frac{\textnormal{Var}^Q
  \left[Z(b)\right]}{\Pr\{ \tau_b < \infty \}^2} = 0.\]
\label{GR2-EFF}
\end{theorem}
\begin{theorem}
  If $\bar{F}(\cdot)$ is regularly varying with index $\alpha > 2,$
  for the choice of $p = (p_k : k \geq 1)$ in \eqref{PK-SEL}:
  \[ \E^Q[\nu_b] \leq \frac{r + o(1)}{\mu (\alpha - 2)}b, \text{ as }
  b \nearrow \infty.\]
\label{GR2-RT}
\end{theorem}
\noindent Proofs of both these results are given later in Section
\ref{PROOFS-KEY-RES}.
\begin{remark}
  From Theorem \ref{GR2-EFF}, we have the vanishing relative error
  property for computing $\Pr\{ \tau_b < \infty\}$ whenever the
  increment random variables $X_n$ have finite mean (irrespective of
  the variance). Therefore we require only
  $o(\epsilon^{-2}\delta^{-1})$ i.i.d replications of $Z(b)$ to arrive
  at estimators that differ relatively at most by $\epsilon$ with
  probability at least $1-\delta.$ Now from Theorem \ref{GR2-RT} we
  conclude that, if the tail index $\alpha > 2$ (in which case the
  increments have finite variance), our importance sampling
  methodology estimates $\Pr \{ \tau_b < \infty\}$ in $O(b)$ expected
  computational effort.
  \label{COMP-EFF-GR2}
\end{remark}
\begin{remark}
  From the conditional limit result in \eqref{COND-LT}, one can infer
  that the values $p_k$ as in \eqref{PK-SEL} roughly match the
  zero-variance probability $\Pr\{ \evk \ | \ \tau_b < \infty\}$
  asymptotically.  For tails $\bar{F}(\cdot)$ with regularly varying
  index $1 < \alpha < 2,$ we have that $\E [ \tau_b \ | \ \tau_b <
  \infty ] = \infty;$ that is, the zero-variance measure itself has
  infinite expected termination time!  Since $p_k$ are assigned a
  value similar to $\Pr \{ \evk \ | \ \tau_b < \infty\},$ one might
  suspect infinite expected termination time for a single run of
  Algorithm 1 as well. As we note later in
  Remark~\ref{remark:remark_infinite_variance} after proof of Theorem
  \ref{GR2-RT}, for $p_k$s as in \eqref{PK-SEL}, this is indeed the
  case.
  \label{ZV-PROB-INF}
\end{remark}

\subsection{Simulation of $\{\tau_b < \infty\}$ - the infinite
  variance case}
\label{SEC-INFIN-VAR}
As indicated in Remark \ref{ZV-PROB-INF}, infinite termination time
for a simulation algorithm is clearly unacceptable. The following
question then is natural: By choosing $p_k$s differently, even if it
means compromising on variance of the estimator, can one achieve
finite expected termination time for the procedure in Section
\ref{SEC-FIN-VAR}?  Before answering this question below, we introduce
a family of tail distributions and their integrated counterparts: for
any $\beta > 2,$ define
\begin{align}
  \Gb(x) := \frac{\bar{F}(x)}{x^{\beta - \alpha}}, \text{ and }
  \Gib(x):= \int_x^{\infty}\Gb(u)du.
\label{GTAILS}
\end{align}
\begin{theorem}
  If the tail $\bar{F}(\cdot)$ is regularly varying with index $\alpha
  \in (1.5,2],$ then for any $\beta \in (2, 2\alpha -1)$,
  \begin{equation}
    p_k = \frac{\Gib(b+\nkm \mu) - \Gib(b+n_k \mu)}{\Gib(b)}, k \geq 1
    \label{PK-SEL-LT2}
  \end{equation}
  yields a family of unbiased estimators $\left( Z(b) = Z_{_K}(b)/\pK:
    b > 0 \right)$ achieving
  \begin{enumerate}
  \item strong efficiency: $\varlimsup_{b \rightarrow \infty}
    \frac{\textnormal{Var}^Q \left[Z(b)\right]}{\Pr\{ \tau_b < \infty
      \}^2} < \infty, \text{ and }$
  \item finite expected termination time: $\E^Q[\nu_b] \leq \frac{r +
      o(1)}{\mu (\beta - 2)}b, \text{ as } b \nearrow \infty$.
  \end{enumerate}
  \label{GR1P5-EFF-RT}
\end{theorem}
\begin{remark}
  Because of the strong efficiency, we need just $O(\epsilon^{-2}
  \delta^{-1})$ i.i.d. replications of $Z(b)$ to achieve the desired
  relative precision.  As in Remark \ref{COMP-EFF-GR2}, due to the
  bound on $\E[\nu_b]$ in Theorem \ref{GR1P5-EFF-RT}, the average
  computational effort for the entire estimation procedure is just
  $O(\epsilon^{-2}\delta^{-1}b).$ It is important to see this
  achievement in the context of Remark \ref{ZV-PROB-INF}: the induced
  measure $Q(\cdot)$ deviates from the zero-variance measure such that
  we get finite expected termination time, but only at the cost of
  losing vanishing relative error property to strong efficiency. Thus
  for the selection of $p_k$s as in \eqref{PK-SEL-LT2}, the suggested
  procedure ends up offering superior performance (in terms of
  computational complexity) compared to the algorithms that tend to
  just approximate the zero-variance measure.
\end{remark}

Given this result, it is difficult not to wonder why the tail index
$\alpha$ should be larger than 1.5 in the statement of Theorem
\ref{GR1P5-EFF-RT}, and what happens when $\alpha \leq 1.5.$ The
following result shows that it is indeed impossible to have both
strong efficiency and finite expected termination time when the tail
index $\alpha < 1.5.$
\begin{theorem}
  If the tail index $\alpha < 1.5,$ there does not exist an assignment
  of $(p_k, n_k : k \geq 1)$ such that both $\E^Q[Z^2(b)]$ and
  $\E^Q[\nu_b]$ are simultaneously finite.
\label{IMPOSS}
\end{theorem}
\begin{remark}
  If the tail index $\alpha = 1.5,$ the possibility of having both
  $\E^Q[Z^2(b)]$ and $\E^Q[\nu_b]$ finite will depend on the slowly
  varying function $L(\cdot)$. As we shall see in the proof of Theorem
  \ref{IMPOSS},
  \[\E^Q[Z^2(b)]\E^Q[\nu_b] =
  \Omega\left(\int_{b^2}^{\infty}\sqrt{u}\bar{F}(u)du) \right),\] as
  $b \nearrow \infty.$ If $L(x) = O((\log x)^{-m})$, $m \geq 2$, the above integral
  is finite, whereas if $L(x) = O(\log x)$ it is infinite; and it
  easily verified that the case of $L(x) = O((\log x)^{-m})$, $m \geq 2$, goes
  through the proof of Theorem \ref{GR1P5-EFF-RT}, thus achieving both
  strong efficiency and finite expected termination time. This
  illustrates the subtle dependence on the associated slowly varying
  function $L(\cdot)$ for the existence of such $p_k$s and $n_k$s.
\end{remark}

As illustrated by the theorem below, for $\alpha \in (1, 1.5],$ we
still have algorithms that demand only $O(b)$ units of expected
computer time if we look for less stringent notions of efficiency.

\begin{theorem}
  If the tail $\bar{F}(\cdot)$ is regularly varying with index $\alpha
  \in (1,1.5],$ then there exists an explicit selection of $p = (p_k:k
  \geq 1)$ such that the family of unbiased estimators $\left( Z(b) :
    b > 0 \right)$ satisfies both:
  \begin{align}
    \varlimsup_{b \rightarrow \infty} \frac{\E^Q
      \left[Z^{1+\gamma}(b)\right]}{\Pr\{ \tau_b < \infty
      \}^{1+\gamma}} &< \infty \text{ for all } \gamma \in
    \left(0,\frac{\alpha-1}{2-\alpha} \right), \text{ and
    } \label{LT1P5-EFF}\\
    \E^Q[\nu_b] &\leq Cb \text{ for some constant C}. \nonumber
  \end{align}
  In particular, for the following selection of $p=(p_k: k \geq 1),$
  \begin{equation}
    p_k = \frac{\Gib(b+\nkm \mu) - \Gib(b+n_k \mu)}{\Gib(b)}, k \geq 1
    \label{PK-SEL-LT1P5}
  \end{equation}
  if $\beta$ is chosen in $(2,\alpha+\gamma^{-1}(\alpha-1)),$ both the
  above inequalities are satisfied.
  \label{LT1P5-EFF-RT}
\end{theorem}
\begin{remark}
  If the estimator $Z(b)$ satisfies \eqref{LT1P5-EFF}, similar to how
  we arrived at \eqref{NSAMP}, it can be shown that
  $O(\epsilon^{-(1+\gamma^{-1})} \delta^{-\gamma^{-1}})$
  i.i.d. replications of $Z(b)$ are enough to produce estimates having
  relative error at most $\epsilon$ with probability at least
  $1-\delta.$ Now according to Theorem \ref{LT1P5-EFF-RT}, the
  expected termination time in each replication is $O(b).$ Thus with
  the $p_k$s chosen as in \eqref{PK-SEL-LT1P5}, we expend just
  $O(\epsilon^{-(1+\gamma^{-1})} \delta^{-\gamma^{-1}}b)$ units of
  computer time on an average, which is still linear in $b.$ The price
  we pay by not adhering to strong efficiency is the worse dependence
  on the parameters $\epsilon$ and $\delta.$
\end{remark}

It is further interesting to note that a vastly different
state-dependent methodology developed using Lyapunov inequalities in
\cite{Blanchet20122994} also hits identical barriers and provides
results similar to ours: They present algorithms that are both
strongly efficient and possess $O(b)$ expected termination time for
the case of tails having index $\alpha > 1.5;$ whereas when $\alpha
\in (1,1.5],$ they provide estimators satisfying \eqref{LT1P5-EFF}
along with $O(b)$ expected termination time of a simulation run.

\section{Proofs of key theorems}
\label{PROOFS-KEY-RES}
For proving Theorems \ref{GR2-EFF}, \ref{GR1P5-EFF-RT} and
\ref{LT1P5-EFF-RT}, which are on the efficiency of estimators $\{Z(b):
b > 0\},$ we first present a result pertaining to the efficiency of
component estimators $\{Z_k(b): k \geq 1\}$. Recall from Section
\ref{SUBSEC-SIM-LOC} that
\begin{align*}
  Z_k(b) := Z_{k,1}(b) + Z_{k,2}(b) + Z_{k,3}(b)
\end{align*}
is an unbiased estimator for $\Pr \{ \evk \},$ and
\[q_k(b) := \sum_{j=\nkm + 1}^{n_k} \bar{F}(b+j\mu).\] To aid the
analysis of second moment of estimators $Z_k(b),$ let $\Pr_k(\cdot)$
denote the composite measure induced due to the simulation of random
variables $Z_{k,j}, j =1,2,3$ independently according to measures
$\Pr_{k,j}, j=1,2,3,$ respectively. Let $\E_k[\cdot]$ denote the
corresponding expectation operator.
\begin{theorem}
  Under Assumptions \ref{LEFT-TAIL-LIGHTER-ASSUMP} and
  \ref{SMOOTHNESS-ASSUMP}, the family of estimators $\{Z_k(b): k \geq
  1, b > 0\}$ satisfies the following as $b \nearrow \infty:$
  \begin{align*} 
    \sup_{k: n_k < b^\eta} \frac{\E_k \left[ Z_k^2(b)\right]
    }{q_k^2(b)} \leq 1+o(1) \text{ and } \sup_{k: n_k \geq b^\eta}
    \frac{ \E_k \left[ Z_k^2(b)\right]}{q_k^2(b)} \leq c
  \end{align*}
  for some $c > 0$ and $\eta > 1.$
  \label{LOC-EFF}
\end{theorem}
We prove Theorem \ref{LOC-EFF} by analysing the second moment of
estimators $Z_{k,1}(\cdot), Z_{k,2}(\cdot)$ and $Z_{k,3}(\cdot)$
separately in the Lemmas \ref{VAR-LEM-1}, \ref{VAR-LEM-2} and
\ref{VAR-LEM-3} below.
\begin{lemma}
  Under Assumption \ref{LEFT-TAIL-LIGHTER-ASSUMP},
  \[ \sup_k \frac{\E_{k,1}\left[ Z_{k,1}^2(b)\right]}{q_k^2(b)} \leq
  1.\]
\label{VAR-LEM-1}
\end{lemma}
\begin{proof}
  Recall that $\Pr_{k,1}(\cdot)$ is the measure resulting due to the
  simulation of increments as in the two-step procedure specified in
  Section \ref{SIM-AK}. Since the quantity $\#\{X_i \geq b+i \mu :
  \nkm < i \leq n_k\}$ is at least 1 when the increments are generated
  from $\Pr_{k,1}(\cdot),$ we have $Z_{k,1}(b) \leq q_k(b).$
  Therefore,
\begin{align}
  \E_{k,1}\left[Z_{k,1}^2(b)\right] \leq q_k^2(b),
  \label{LOC-VAR-1-I}
\end{align}
which proves the claim.
\end{proof}

For a similar analysis on the second moment of estimators $Z_{k,2}(b)$
and $Z_{k,3}(b),$ we need the following results which are proved in
the appendix.
\begin{lemma}
  Under Assumption \ref{LEFT-TAIL-LIGHTER-ASSUMP}, there exists a
  constant $c_1 > 1$ such that $\exp(n_k\Lambda_k(\theta_k)) \leq c_1$
  for all $k,b.$
\label{NORM-TERM}
\end{lemma}
\begin{lemma}
  Under Assumption \ref{LEFT-TAIL-LIGHTER-ASSUMP}, there exists a
  positive constant $c_2$ such that,
  \[ \sup_{k \geq 1, b >0} \frac{n_k\bar{F}(b+\nkm\mu)}{q_k(b)} \leq
  c_2.\]
\label{OTHER-TERM}
\end{lemma}
\begin{proposition}
  Under Assumption \ref{LEFT-TAIL-LIGHTER-ASSUMP},
  \[ \sup_{k \geq 1} \left|\frac{\Pr \left\{\evk, A_k
      \right\}}{q_k(b)} -1\right| =
  O\left(b^{-\frac{\alpha-1}{2\alpha}}\right),\] $\text{ as } b
  \nearrow \infty.$
  \label{AK-ASYMP}
\end{proposition}
\begin{lemma}
  Under Assumptions \ref{LEFT-TAIL-LIGHTER-ASSUMP} and
  \ref{SMOOTHNESS-ASSUMP}, there exist constants $\eta > 1$ and $c_3$
  such that,
  \begin{align*} 
    \sup_{k: n_k < b^\eta} \frac{\Pr
      \{\evk,\bar{A}_k\} }{q_k(b)} &= o(1) \text{ and }\\
    \sup_{k: n_k \geq b^\eta} \frac{\Pr \{\evk,\bar{A}_k\} }{q_k(b)}
    &\leq c_3,
\end{align*}
as $b \nearrow \infty.$
\label{RES-TERM-ASYMP}
\end{lemma}
\noindent Using Lemmas \ref{NORM-TERM}, \ref{OTHER-TERM} and
\ref{RES-TERM-ASYMP}, we now present an asymptotic analysis on the
second moment of estimators $Z_{k,2}(\cdot)$ and $Z_{k,3}(\cdot).$
\begin{lemma}
  Under Assumptions \ref{LEFT-TAIL-LIGHTER-ASSUMP} and
  \ref{SMOOTHNESS-ASSUMP}, as $b \nearrow \infty,$
  \begin{align*} 
    \sup_{k: n_k < b^\eta} \frac{\E_{k,2} \left[ Z_{k,2}^2(b)\right]
    }{q_k^2(b)} = o(1) \text{ and } \sup_{k: n_k \geq b^\eta} \frac{
      \E_{k,2} \left[ Z_{k,2}^2(b)\right]}{q_k^2(b)} \leq c_4
\end{align*}
for some positive constant $c_4.$
\label{VAR-LEM-2}
\end{lemma}
\begin{proof}
  Since $\tau_b \leq n_k$ on the event $\{\evk\},$
\begin{align*}
  \exp\left(\tau_b\Lambda_k(\theta_k)\right) \mathbb{I}(\evk,B_k) \leq
  c_1,
\end{align*}
because of Lemma \ref{NORM-TERM}. Further note that
$\theta_kS_{\tau_b} \geq -\log(n_k\bar{F}(b+\nkm\mu))$ on $\{ \evk\}.$
Therefore from \eqref{LOC-EST-2},
\begin{align*}
  Z_{k,2}(b) \leq
  c_1\left(n_k\bar{F}(b+\nkm\mu)\right)\mathbb{I}(\evk,B_k), \text{
    for all } k.
\end{align*}
Now changing the expectation operator in the evaluation of second
moment of the estimator results in the following bound: for all $k,$
\begin{align*}
  \E_{k,2}\left[ Z_{k,2}^2(b) \right] = \E \left[ Z_{k,2}(b) \right]
  \leq c_1\left(n_k\bar{F}(b+\nkm \mu)\right)\Pr \{ \evk,B_k \}.
\end{align*}
\[ \hspace{-30pt} \text{Therefore } \frac{\E_{k,2} \left[
    Z_{k,2}^2(b)\right] }{q_k^2(b)} \leq
c_1\frac{\left(n_k\bar{F}(b+\nkm \mu)\right)}{q_k(b)}\frac{\Pr \{
  \evk,\bar{A}_k \}}{q_k(b)}.\] Then it follows from Lemmas
\ref{OTHER-TERM} and \ref{RES-TERM-ASYMP} that, as $b \nearrow
\infty,$
\begin{align*}
  \sup_{k: n_k < b^\eta} \frac{\E_{k,2} \left[ Z_{k,2}^2(b)\right]
  }{q_k^2(b)} = o(1), \text{ and } \sup_{k: n_k \geq b^\eta}
  \frac{\E_{k,2} \left[ Z_{k,2}^2(b)\right] }{q_k^2(b)} \leq c_1c_2c_3
  =: c_4 < \infty,
\end{align*}
thus proving the claim.
\end{proof}
\begin{lemma}
  Under Assumptions \ref{LEFT-TAIL-LIGHTER-ASSUMP} and
  \ref{SMOOTHNESS-ASSUMP}, as $b \nearrow \infty,$
  \begin{align*} 
    \sup_{k: n_k < b^\eta} \frac{\E_{k,3} \left[ Z_{k,3}^2(b)\right]
    }{q_k^2(b)} = o(1) \text{ and } \sup_{k: n_k \geq b^\eta} \frac{
      \E_{k,3} \left[ Z_{k,3}^2(b)\right]}{q_k^2(b)} \leq c_5
\end{align*}
for some positive constant $c_5.$
\label{VAR-LEM-3}
\end{lemma}
\begin{proof}
  When the increments are generated as prescribed in the two-step
  procedure in Section \ref{SIM-ABK}, we have $\#\{ X_i \geq b+ \nkm
  \mu : 1 < i \leq n_k \} \geq 1,$ and hence,
  \[ Z_{k,3}(b) \leq n_k \bar{F}(b+\nkm \mu)\mathbb{I}\left(\evk,
    \bar{A}_k\cap\bar{B}_k \right). \] Now a bound on the second
  moment of the estimator can be obtained as before:
\begin{align*}
  \E_{k.3}\left[Z_{k,3}^2(b) \right] = \E\left[ Z_{k.3}(b) \right]
  \leq n_k\bar{F}(b+\nkm \mu)\Pr \left\{\evk, \bar{A}_k\cap\bar{B}_k
  \right\}.
\end{align*}
\[ \hspace{-30pt} \text{Therefore } \frac{\E_{k,3} \left[
    Z_{k,3}^2(b)\right] }{q_k^2(b)} \leq \frac{\left(n_k\bar{F}(b+\nkm
    \mu)\right)}{q_k(b)}\frac{\Pr \{ \evk,\bar{A}_k \}}{q_k(b)}.\]
Then it follows from Lemmas \ref{OTHER-TERM} and \ref{RES-TERM-ASYMP}
that, as $b \nearrow \infty,$
\begin{align*}
  \sup_{k: n_k < b^\eta} \frac{\E_{k,3} \left[ Z_{k,3}^2(b)\right]
  }{q_k^2(b)} = o(1), \text{ and } \sup_{k: n_k \geq b^\eta}
  \frac{\E_{k,3} \left[ Z_{k,3}^2(b)\right] }{q_k^2(b)} \leq c_2c_3 =:
  c_5 < \infty,
\end{align*}
thus establishing the claim.
\end{proof}

\paragraph{Proof of Theorem \ref{LOC-EFF}}
Since $\{Z_{k,i}(b):i=1,2,3\}$ are independent, for $i \neq 1,$
  \begin{align*}
    \frac{\E_{k}\left[Z_{k,1}(b)Z_{k,i}(b)\right]}{q_k^2(b)} &=
    \frac{\E_{k,1}\left[Z_{k,1}(b)\right]}{q_k(b)}
    \frac{\E_{k,i}\left[Z_{k,i}(b)\right]}{q_k(b)}\\
    &\leq
    \frac{\Pr\{\evk,A_k\}}{q_k(b)}\frac{\Pr\{\evk,\bar{A}_k\}}{q_k(b)}.
  \end{align*}
  Then from Proposition \ref{AK-ASYMP} and Lemma \ref{RES-TERM-ASYMP},
  we have that as $b \nearrow \infty,$
  \[ \sup_{k:n_k < b^\eta}
  \frac{\E_{k}\left[Z_{k,1}(b)Z_{k,i}(b)\right]}{q_k^2(b)} = o(1),
  \text{ and } \sup_{k:n_k \geq b^\eta}
  \frac{\E_{k}\left[Z_{k,1}(b)Z_{k,i}(b)\right]}{q_k^2(b)} <
  \infty. \] Similarly from Lemma \ref{RES-TERM-ASYMP}, as $b \nearrow
  \infty,$
  \[ \sup_k \frac{\E_{k}\left[Z_{k,2}(b)Z_{k,3}(b)\right]}{q_k^2(b)}
  \leq \sup_k \frac{\Pr\{\evk,\bar{A}_k\}^2}{q_k^2(b)} = o(1).
  \]
  Since $Z_k(b) = Z_{k,1}(b) + Z_{k,2}(b) + Z_{k,3}(b),$ we have
  \[\E_k\left[Z_k^2(b)\right] =
  \sum_{i,j=1}^3\E_k\left[Z_{k,i}(b)Z_{k,j}(b)\right].\] Combining
  above observations with the results of Lemmas \ref{VAR-LEM-1},
  \ref{VAR-LEM-2} and \ref{VAR-LEM-3}, we conclude that as $b \nearrow
  \infty,$
  \begin{align*}
    \sup_{k: n_k < b^\eta}\frac{\E_k\left[Z_k^2(b)\right]}{q_k^2(b)}
    \leq 1+ o(1) \text{ and } \sup_{k: n_k \geq
      b^\eta}\frac{\E_k\left[Z_k^2(b)\right]}{q_k^2(b)} \leq c
  \end{align*}
  for some positive constant $c.$ \hfill{$\Box$} 

\vspace{10pt} 
\noindent The following uniform bounds will be useful:
\begin{lemma}
  For all $k \geq 1,$
  \[ q_k(b) \leq \frac{1}{\mu}\left(\bar{F}_I(b+\nkm \mu) -
    \bar{F}_I(b+n_k \mu)\right).\] Further as $b \nearrow \infty,$
  \[ q_k(b) \geq (1-o(1))\frac{1}{\mu}\left(\bar{F}_I(b+\nkm \mu) -
    \bar{F}_I(b+n_k \mu)\right), \] uniformly in $k.$
\label{QKB-FINT-REL-LEM}
\end{lemma}
\begin{proof}
  For any $k \geq 1,$
  \begin{align*}
    q_k(b) = \sum_{i=\nkm + 1}^{n_k}\bar{F}(b+ i \mu)\leq \sum_{i=\nkm
      + 1}^{n_k}\int_{i-1}^{i}\bar{F}(b+u\mu)du =
    \int_{\nkm}^{n_k}\bar{F}(b+u\mu)du.
  \end{align*}
  Changing variables from $u$ to $v=b+u\mu$ results in,
  \begin{align*}
    q_k(b) \leq \frac{1}{\mu}\int_{b+\nkm\mu}^{b+n_k\mu} \bar{F}(v)dv,
  \end{align*}
  which establishes the upper bound because $\bar{F}_I(x) :=
  \int_x^\infty\bar{F}(u)du.$ 

  For the lower bound, see that
  \begin{align*}
    q_k(b) &= \sum_{i=\nkm + 1}^{n_k}\bar{F}(b+ i \mu) \geq
    \sum_{i=\nkm + 1}^{n_k}\int_{i}^{i+1}\bar{F}(b+u\mu)du\\
    &= \int_{\nkm+1}^{n_k+1}\bar{F}(b+u\mu)du.
  \end{align*}
  Now after changing variables from $u$ to $v=b+u\mu,$ we use the
  long-tailedness of $\bar{F}_I(\cdot)$ to see that, given $\epsilon >
  0,$ for large values of $b,$
  \begin{align*}
    q_k(b) &\geq \frac{1}{\mu}\left(\bar{F}_I(b+ (\nkm +1)\mu) -
      \bar{F}_I(b+ (n_k+1) \mu)\right)\\
    &\geq (1-\epsilon)\frac{1}{\mu}\left( \bar{F}_I(b+ \nkm\mu) -
      \bar{F}_I(b+ n_k \mu) \right)
  \end{align*}
  for all $k.$
\end{proof}

\paragraph{Proof of Theorem \ref{GR2-EFF}}
Recall that the overall estimator is, \[Z(b) =
\frac{Z_K(b)}{p_{_K}},\] where $p_k$ is as in \eqref{PK-SEL}. Second
moment of the estimator $Z(b)$ is bounded as below:
\begin{align}
  \E^Q[Z^2(b)] &= \E^Q\left[ \left(\frac{Z_K(b)}{p_{_K}}
    \right)^2\right]
  \nonumber\\
  &= \E^Q\left[ \E^Q\left[\frac{Z_K^2(b)}{q^2_{_K}(b)}
      \frac{q^2_{_K}(b)}{p^2_{_K}}; n_{_K} < b^\eta
      \left.\frac{}{}\right| K \right]\right] \nonumber\\
  &\hspace{30pt}+ \E^Q\left[
    \E^Q\left[\frac{Z_K^2(b)}{q^2_{_K}(b)}\frac{q^2_{_K}(b)}{p^2_{_K}};
      n_{_K} \geq b^\eta \left.\frac{}{}\right| K
    \right]\right] \label{COND-INT}
\end{align}
From the definition of $p_k$ and Lemma \ref{QKB-FINT-REL-LEM}, we have
$q^2_{_K}(b) \leq \bar{F}_I^2(b)p^2_{_K}.$ Combining this with Theorem
\ref{LOC-EFF} it follows that,
\begin{align*}
  \frac{\E^Q[Z^2(b)]}{\bar{F}_I^2(b)} &\leq
  \E^Q\left[\E^Q\left[\frac{Z_K^2(b)}{q^2_{_K}(b)}; n_{_K} < b^\eta
      \left.\frac{}{}\right| K \right]\right] +
  \E^Q\left[\E^Q\left[\frac{Z_K^2(b)}{q^2_{_K}(b)}; n_{_K} \geq b^\eta
      \left.\frac{}{}\right| K \right]\right]\\
  &\leq 1 + o(1) + c \Pr\{n_{_K} \geq b^\eta\}\\
  &\leq 1 + o(1) +
  O\left(\frac{\bar{F}_I(b+b^\eta)}{\bar{F}_I(b)}\right) = 1+o(1),
\end{align*}
as $b \nearrow \infty.$ The last inequality follows from observing
that $\Pr\{ n_{_K} \geq b^\eta\} = \sum_{k:n_k \geq b^\eta}p_k.$ Since
$\eta > 1,$ we have the asymptotically vanishing relative error
property of the estimators $(Z(b): b > 0).$ \hfill\(\Box\)

\paragraph{Proof of Theorem \ref{GR2-RT}}
Recall that $\nu_b$ denotes the maximum of indices of the increment
random variables ($X_i$s) considered for simulation in a particular
simulation run. From the sampling procedures in Section
\ref{SUBSEC-SIM-LOC}, it is clear that $\nu_b \leq n_{_K}.$ Therefore,
  \begin{align}
    \E^Q [ \nu_b] &\leq \sum_{k \geq 1} p_kn_k \nonumber\\
    &= rp_1 + \sum_{k \geq 2}r^k p_k \nonumber\\
    &= \frac{1}{\Fi(b)}\left(r\int_{b}^{b+r\mu} \bar{F}(u)du + \sum_{k
        \geq 1} r^{k+1} \int_{b+r^k\mu}^{b+r^{k+1}\mu}\bar{F}(u)du
    \right).
    \label{NSTEPS-INT}
  \end{align}
  \begin{align*}
    \text{Since } r^k \int_{b+r^k \mu}^{b+r^{k+1}\mu}&\bar{F}(u)du
    = \frac{b + r^{k}\mu
      -b}{\mu}\int_{b+r^{k}\mu}^{b+r^{k+1}\mu}\bar{F}(u)du\\
    &\leq
    \frac{1}{\mu}\left(\int_{b+r^{k}\mu}^{b+r^{k+1}\mu}u\bar{F}(u)du -
      b\int_{b+r^{k}\mu}^{b+r^{k+1}\mu}\bar{F}(u)du \right),
  \end{align*}
  \begin{align}
    \text{we write }\sum_{k \geq 1}r^{k+1}&
    \int_{b+r^{k}\mu}^{b+r^{k+1}\mu}\bar{F}(u)du \nonumber\\
    &\leq \frac{r}{\mu}
     \sum_{k \geq
        1}\left(\int_{b+r^{k}\mu}^{b+r^{k+1}\mu}u\bar{F}(u)du -
      b\int_{b+r^{k}\mu}^{b+r^{k+1}\mu}\bar{F}(u)du \right) \nonumber\\
    &=\frac{r}{\mu}\left(\int_{b+r\mu}^{\infty}u\bar{F}(u)du -
      \int_{b+r\mu}^{\infty}\bar{F}(u)du \right) \label{DOUB-INT-TAIL}\\
    &\leq \frac{r + o(1)}{\mu} \left( \frac{(b+r\mu)^2}{\alpha-2} -
      b\frac{b+r\mu}{\alpha-1}\right)\bar{F}(b+r\mu), \nonumber\\
    &= \frac{r+o(1)}{\mu (\alpha-1)(\alpha-2)} b^2\bar{F}(b), \text{
      as } b \nearrow \infty.  \nonumber
  \end{align}
  where the penultimate step follows from Karamata's theorem (see
  \eqref{KARAMATA}), and the final step just uses long-tailed nature
  of $\bar{F}(\cdot).$ Also note that: $\int_{b}^{b+r\mu} \bar{F}(u)du
  \leq r\mu \bar{F}(b),$ and by application of Karamata's theorem, we
  have $\Fi(b) \sim {bF(b)}/(\alpha-1),$ as $b \nearrow \infty.$
  Therefore from \eqref{NSTEPS-INT},
  \begin{align*}
    \E^Q [ \nu_b] &\leq \frac{r+o(1)}{\mu(\alpha-2)} b, \text{ as } b
    \nearrow \infty,
  \end{align*}
  thus yielding the required bound on the expected termination
  time. \hfill\(\Box\)
\begin{remark} \label{remark:remark_infinite_variance}
  Similar to how we arrived at \eqref{DOUB-INT-TAIL}, lower bounds can
  be obtained to show that $\E^Q[\nu_b] = \Omega \left(
    \int_{b}^{\infty} u \bar{F}(u)du\right).$ If the tail index
  $\alpha < 2, \int_{b}^{\infty} u \bar{F}(u)du$ turns out to be
  infinite, and subsequently $\E^Q[\nu_b] = \infty.$ Though the
  assignment of $p_k$s in \eqref{PK-SEL} yields vanishing relative
  error for any $\alpha > 1,$ it fails to provide algorithms which
  have finite expected termination time when the increment random
  variables $X$ have infinite variance (e.g., when $\alpha <
  2$), thus making this choice of $p_k$ not suitable for practice.
\end{remark}

\paragraph{Proof of Theorem \ref{GR1P5-EFF-RT}} We obtain upper
bounds for both the variance of the estimator $Z(b)$ and the expected
termination time.\\
1. \textit{Variance of $Z(b)$:} Since $Q(K=k) = p_k,$
\begin{align}
  \E^Q[Z^2(b)] &= \E^Q\left[ \frac{Z^2_K(b)}{\pK^2}\right] = \sum_k
  p_k \frac{\E^Q[Z_k^2(b)]}{p_k^2}\\
  &= \sum_k \frac{\E^Q[Z_k^2(b)]}{q^2_k(b)} \frac{q^2_k(b)}{p_k}.
  \label{INF-INTER1}
\end{align}
Following Lemma \ref{QKB-FINT-REL-LEM} and the assignment of $p_k$s as
in \eqref{PK-SEL-LT2}, we can write,
\begin{align*}
  \frac{q_k(b)}{p_k} &\leq \frac{ \Fi(b+\nkm \mu) -
    \Fi(b+n_k\mu)}{\Gib(b+\nkm \mu) - \Gib(b+n_k\mu)} \Gib(b).
\end{align*}
To obtain an upper bound, we note the following:
\begin{align*}
  \Fi(b+\nkm \mu) - \Fi(b+n_k\mu) &= \int_{b+\nkm
    \mu}^{b+n_k\mu}\bar{F}(u)du\\
  &\leq (n_k-\nkm)\mu \bar{F}(b+\nkm \mu),\\
  \Gib(b+\nkm \mu) - \Gib(b+n_k\mu) &= \int_{b+\nkm
    \mu}^{b+n_k\mu}\Gb(u)du\\
  &\geq (n_k-\nkm)\mu \Gb(b+n_k\mu),\text{
    and }\\
   \frac{\Gb(b+\nkm \mu)}{\Gb(b+n_k \mu)} 
   &\leq r^{\beta}+o(1), \text{ as } b \nearrow \infty.
 \end{align*}
 The last inequality follows by observing that $b+n_k \mu \leq
 r(b+\nkm \mu)$ and subsequently from the regularly varying nature of
 $\Gb(\cdot).$ Therefore as $b \nearrow \infty,$
\begin{align}
  \frac{q_k(b)}{p_k} &\leq \frac{\Gb(b+\nkm \mu)}{\Gb(b+n_k
    \mu)}\frac{\bar{F}(b+\nkm
    \mu)}{\Gb(b+\nkm \mu)} \Gib(b)\nonumber\\
  &= (r^{\beta}+ o(1)) (b+\nkm \mu)^{\beta-\alpha}\Gib(b),
  \label{FRAC-INTER}
\end{align}
for all $k,$ because ${\bar{F}(x)}/{\Gb(x)} = x^{\beta-\alpha}.$
Combining this with Theorem \ref{LOC-EFF}, it follows from
\eqref{INF-INTER1} that 
\begin{align*}
  \E^Q[Z^2(b)] &\leq (cr^{\beta}+o(1))\Gib(b) \sum_k (b+\nkm
  \mu)^{\beta-\alpha}q_k(b)\\
  &\leq (cr^{\beta} + o(1)) \Gib(b)\sum_k (b+\nkm \mu)^{\beta-\alpha}
  \int_{b+\nkm \mu}^{b+n_k \mu} \bar{F}(u)du,\\
  &\leq (cr^{\beta}+o(1))\Gib(b)\sum_k \int_{b+\nkm \mu}^{b+n_k \mu}
  u^{\beta-\alpha}\bar{F}(u)du\\
  &\leq (cr^{\beta}+o(1)) \Gib(b)
  \int_{b}^{\infty}u^{\beta-\alpha}\bar{F}(u)du
\end{align*}
as $b \nearrow \infty.$ Since $2\alpha-\beta > 1,$ it follows from
Karamata's theorem (cf. \eqref{KARAMATA}) that
\[ \E^Q[Z^2(b)] \leq (cr^{\beta}+o(1))
\Gib(b)b^{\beta-\alpha+1}\frac{\bar{F}(b)}{2\alpha-\beta-1}, \text{ as
} b \nearrow \infty.\] Further $(\alpha-1)\Fi(b) \sim b\bar{F}(b)$ and
$b^{\beta-\alpha}\Gib(b) \sim \Fi(b),$ as $b \nearrow \infty.$
Therefore,
\[ \varlimsup_{b \rightarrow \infty} \frac{\E^Q[Z^2(b)]}{\Fi^2(b)}
\leq \frac{(\alpha-1)cr^{\beta}+o(1)}{2\alpha-\beta-1} < \infty.\] Now
since $\Pr \{\tau_b < \infty \} \sim \mu^{-1}\Fi(b),$ we have strong
efficiency.\\\\
2. \textit{Expected termination time:} Since $\nu_b \leq n_{_K},
\E^Q[\nu_b] \leq \E^Q[n_{_K}] = \sum_k p_kn_k.$ For the choice of
$p_k$ in \eqref{PK-SEL-LT2}, following exactly the same steps in the
proof of Theorem \ref{GR2-RT}, we arrive at:
\begin{align*}
  \E^Q[\nu_b] &\leq \frac{r}{\mu}\left( \mu \int_b^{b+r\mu}\Gb(u)du +
    \int_{b+r\mu}^{\infty} u\Gb(u)du -
    b\int_{b+r\mu}^{\infty}\Gb(u)du\right).
\end{align*}
Since $\Gb(\cdot)$ is regularly varying with tail index larger than 2,
by application of Karamata's theorem, we have:
\[ \int_{b+r\mu}^{\infty} u\Gb(u)du \sim
\frac{(b+r\mu)^2}{\beta-2}\Gb(b+r\mu),\]
which would not have been the case if we had persisted with using
$\Fi(\cdot)$ instead of $\Gib(\cdot)$ for $p_k.$ Again following the
remaining steps in the proof of Theorem \ref{GR2-RT}, we conclude
that:
\[\E^Q [ \nu_b] \leq \frac{r+o(1)}{\mu(\beta-2)} b, \text{ as } b
\nearrow \infty, \] thus yielding finite termination time even when
the zero-variance measure fails to offer this  desirable
property. \hfill\(\Box\)

\paragraph{Proof of Theorem \ref{IMPOSS}}
Since $Q(K=k) = p_k,$ see that:
\[\E^Q[Z^2(b)] = \E^Q\left[\frac{Z^2_K(b)}{\pK^2}\right]= \sum_k
\frac{\E^Q[Z_k^2(b)]}{p_k} \geq \sum_k \frac{\Pr \{ \evk \}^2}{p_k},\]
because of Jensen's inequality. To arrive at a contradiction, let us
assume that both $\E^Q[Z^2(b)]$ and $\E^Q[\nu_b]$ are finite. Then,
\begin{align}
  \E^Q[Z^2(b)]\E^Q[\nu_b] &\geq \left(\sum_k \frac{\Pr \{ \evk
      \}^2}{p_k} \right) \left(\sum_k p_k n_k
  \right) \nonumber\\
  &\geq \left( \sum_k \frac{\Pr \{ \evk \}}{\sqrt{p_k}}\cdot
    \sqrt{p_kn_k} \right)^2 \nonumber\\
  &= \left( \sum_k \sqrt{n_k} \Pr \{ \evk
    \}\right)^2. \label{IMPOSS-INTER}
\end{align}
where the penultimate step follows from Cauchy-Schwarz
inequality. Then from Proposition \ref{AK-ASYMP} and Lemma
\ref{QKB-FINT-REL-LEM}, it is immediate that
\begin{align*}
  \sum_k \sqrt{n_k}\ \Pr \{ \evk \} &\geq (1-o(1)) \sum_k
  \sqrt{n_k}\ q_k(b)\\
  &\geq (1-o(1)) \sum_k \sqrt{n_k}
  \int_{\nkm \mu}^{n_k\mu}  \bar{F}(b + u)du\\
  &\geq \frac{1-o(1)}{\sqrt{\mu}} \sum_k \int_{\nkm
    \mu}^{n_k\mu} \sqrt{u}\bar{F}(b + u)du\\
  &= \frac{1-o(1)}{\sqrt{\mu}} \int_{0}^{\infty} \sqrt{u}\bar{F}(b +
  u)du.
 \end{align*}
 Now it can be seen easily that the RHS is finite only when $\alpha
 \geq 1.5,$ via the following change of variable and the subsequent
 integration of the resulting regularly varying tail:
\begin{align*}
  \int_{0}^{\infty} \sqrt{u}\bar{F}(b + u)du &= \int_{b}^{\infty}
  \sqrt{u-b}\bar{F}(u)du\\
  &\geq \int_{b^2}^{\infty}
  \sqrt{u}\cdot\sqrt{1-\frac{b}{u}}\bar{F}(u)du\\
  &\geq \sqrt{1-\frac{1}{b}} \int_{b^2}^{\infty}\sqrt{u}\bar{F}(u)du,
\end{align*}
which cannot be finite if $\alpha < 1.5,$ thus arriving at the desired
contradiction. Therefore from \eqref{IMPOSS-INTER}, we conclude that
we cannot have both the second moment of $Z(b)$ and the expected
termination time $\E^Q[\nu_b]$ to be simultaneously finite if the
tail index $\alpha < 1.5.$ \hfill\(\Box\)\\

\paragraph{Proof of Theorem \ref{LT1P5-EFF-RT}}
The proof is similar to that of Theorem \ref{GR1P5-EFF-RT}, and we
provide only an outline of the steps involved. Since $Q(K=k)=p_k,$
\begin{align*}
  \E^Q[Z^{1+\gamma}(b)] &= \E^Q\left[
    \frac{Z^{1+\gamma}_K(b)}{p_{_K}^{1+\gamma}} \right] = \sum_k
  \frac{\E_k \left[Z^{1+\gamma}_k(b)\right]}{p_k^{1+\gamma}}p_k\\
  &\leq \sum_k \left( \frac{\E_k
      \left[Z^2_k(b)\right]}{q_k^2(b)}\right)^{\frac{1+\gamma}{2}}
  \left(\frac{q_k(b)}{p_k}\right)^\gamma q_k(b)
\end{align*}
Now from Theorem \ref{LOC-EFF} and \eqref{FRAC-INTER} , following the
routine calculation in the proof of Theorem \ref{GR1P5-EFF-RT}, we
deduce that
\begin{align*}
  \E^Q[Z^{1+\gamma}(b)] &\leq \left(c^{\frac{1+\gamma}{2}} r^{\beta
      \gamma}+o(1)\right) \left( \Gib(b) \right)^\gamma \sum_k \left(
    b + \nkm
    \mu\right)^{\gamma (\beta - \alpha)} q_k(b)\\
  &\leq \left(c^{\frac{1+\gamma}{2}}r^{\beta \gamma}+o(1)\right)
  \left( \Gib(b) \right)^\gamma
  \int_b^{\infty}u^{\gamma(\beta-\alpha)}\bar{F}(u)du,
\end{align*}
as $b \nearrow \infty.$ Since $\beta$ is smaller than
$\alpha+\gamma^{-1}(\alpha-1)$ as in the statement of Theorem
\ref{LT1P5-EFF-RT}, the tail index of the integrand, $\alpha - \gamma
(\beta-\alpha) > 1.$ Therefore we can apply Karamata's theorem to
conclude that
\begin{align*}
  \E^Q[Z^{1+\gamma}(b)] &\leq \left(c^{\frac{1+\gamma}{2}}r^{\beta
      \gamma}+o(1)\right) \left( \Gib(b) \right)^\gamma
  \frac{b^{\gamma (\beta - \alpha)+1}}{\alpha -
    \gamma(\beta-\alpha)-1}\bar{F}(b), \text{ as } b \nearrow \infty.
\end{align*}
Now observing that $(\alpha-1)\Fi(b) \sim b\bar{F}(b),
b^{\beta-\alpha}\Gib(b) \sim \Fi(b),$ and $\Pr \{ \tau_b < \infty\}
\sim \mu^{-1}\Fi(b)$ as $b \nearrow \infty,$ we have:
\[ \varlimsup_{b \rightarrow \infty}
\frac{\E^Q[Z^{1+\gamma}(b)]}{\Pr\{\tau_b < \infty\}^{1+\gamma}} \leq
\frac{\mu^2 (\alpha-1) c^{\frac{1+\gamma}{2}}r^{\beta \gamma}+o(1)
}{\alpha - \gamma(\beta-\alpha)-1} < \infty.\] Since $\beta$ is
ensured to be larger than 2, the same proof for $\E^Q[\nu_b] = O(b)$
goes through. \hfill\(\Box\)

\section{Simulation of $\tau_b < \tau$}
\label{SEC-BCYC-SIM}
Let $X,X_1,X_2,\ldots$ be an iid collection of random variables
satisfying the following assumption:
\begin{assumption}
  The tail probabilities of $X$ are given by $\bar{F}(x) := \Pr \{ X >
  x \} = x^{-\alpha}L(x),$ for some slowly varying function $L(\cdot)$
  and $\alpha > 2.$ Further, $\mu := -\E X > 0.$
\label{REG-VAR-ASSUMP}
\end{assumption} 
As in the Sections \ref{SEC-LD} and \ref{SEC-SIM-METH}, let $S_0 = 0,
S_n = X_1 + \ldots + X_n,$ for $n \geq 1.$ Further, let $M_n = \max_{k
  \leq n}S_k, \tau = \inf \{ n \geq 1: S_n \leq 0\}$ and $\tau_b =
\inf \{ n \geq 1: S_n > x \}$ for $b > 0.$ Our aim is to simulate the
tail probabilities of busy cycle maximum $M_{\tau}.$ In other words, we
aim to simulate $\Pr \{ M_\tau > b\} = \Pr \{\tau_b < \tau\}$
efficiently, as $b \nearrow \infty.$ Under Assumption
\ref{REG-VAR-ASSUMP}, it is well-known that (see, for example, Theorem
2.1 of \cite{asmussen1998})
\begin{align}
  \label{BCYC-ASYMP}
  \Pr \{\tau_b < \tau\} \sim \E \tau \bar{F}(b), \text{ as } b
  \nearrow \infty.
\end{align}

As in the simulation of $\{ S_n > b\},$ we partition the probability
of interest into dominant and residual components as below:
\[ \Pr \left\{ \tau_b < \tau \right\} = \Pr \left\{ \tau_b < \tau,
  \max_{k \leq \tau_b} X_k > b \right\} + \Pr \left\{ \tau_b < \tau,
  \max_{k \leq \tau_b} X_k \leq b \right\}.\] Since $S_n > 0$ for all
$n < \tau,$ the first component has a simple representation:
\begin{align*}
  \Pr \left\{ \tau_b < \tau, \max_{k \leq \tau_b} X_k > b \right\} &=
  \Pr \left\{ \tau_b < \tau,  X_{\tau_b} > b \right\}\\
  &= \sum_{n=1}^\infty \Pr \left\{ S_i \in (0,b] \text{ for }
    i=1,\ldots,n-1, X_n > b \right\}\\
  &= \sum_{n=1}^\infty \Pr \left\{ S_i \in (0,b] \text{ for }
    i=1,\ldots,n-1 \right\} \bar{F}(b)\\
  &= \bar{F}(b) \sum_{n=1}^\infty \Pr \left\{ \tau_b \wedge \tau > n-1
  \right\} \\
  &= \E \left[ \tau_b \wedge \tau \right] \bar{F}(b).
\end{align*}
Therefore to estimate $\Pr \left\{ \tau_b < \tau, \max_{k \leq \tau_b}
  X_k > b \right\},$ we draw samples of increments $X_n$ naively from
the distribution $F(\cdot)$, and compute the following as the
estimator:
\begin{align}
  \label{DOM-EST-BCYC}
  \Zd (b) := (\tau_b \wedge \tau) \bar{F}(b).
\end{align}
Now it is straightforward to see that 
\begin{align*}
  \E \left[ \Zd \right] &= \Pr \left\{ \tau_b < \tau, \max_{k \leq
      \tau_b} X_k > b \right\},\\
  \textnormal{Var} \left[ \Zd \right] &= \textnormal{Var} \left[\tau_b
  \wedge \tau \right] \bar{F}^2(b),
\end{align*}
and hence, due to \eqref{BCYC-ASYMP} and monotone convergence,
\begin{align}
  \label{VAR-DOM-BCYC}
  \varlimsup_{b \rightarrow \infty} \frac{\textnormal{Var} \left[ \Zd
    \right]}{ \Pr \left\{ \tau_b < \tau \right\}^2} = \varlimsup_{b
    \rightarrow \infty} \frac{\textnormal{Var} \left[\tau_b \wedge
      \tau \right]}{\E \left[\tau \right]^2} =
  \frac{\textnormal{Var}\left[\tau\right]}{\E \left[ \tau \right]^2}.
\end{align}
To estimate the residual probability $\Pr \left\{ \tau_b < \tau,
  \max_{k \leq \tau_b} X_k \leq b \right\},$ we perform exponential
twisting as in Section \ref{SUBSEC-SIM-AR}. Draw samples of $\{X_n : n
\leq \tau_b \wedge \tau \}$ independently from $F_\theta(\cdot)$ given
by:
\begin{align}
  \label{RES-EXP-TWIST-DIST}
  \frac{dF_\theta}{dF}(x) = \exp \left( \theta_b x -
    \Lambda_b(\theta_b)\right)\mathbf{1}(x \leq b),
\end{align}
where
\begin{align*}
  \Lambda_b(\theta) &:= \log \left( \int_{-\infty}^b \exp \left(
      \theta x\right) F(dx) \right) \text{ for } \theta > 0, \text{ and }\\
  \theta_b &:= -\frac{\log b\bar{F}(b)}{b}.
\end{align*}
Then the resulting estimator is given by
\begin{align}
  \label{RES-EST-BCYC}
  \Zr (b) := \exp \left( - \theta_b S_{\tau_b} + \tau_b
    \Lambda_b(\theta_b)\right)\mathbb{I} \left(\tau_b < \tau, \max_{k
      \leq \tau_b} X_k \leq b \right).
\end{align}
For proving efficiency results of $\Zr (b),$ we shall need the
following results that are proved in the appendix.
\begin{proposition}
  Under Assumption \ref{REG-VAR-ASSUMP},
  \[\Pr \left\{ \tau_b < \tau,  \max_{k \leq \tau_b} X_k \leq b \right\} =
  O\left( \frac{\bar{F}(b)}{b}\right), \text{ as } b \nearrow
  \infty.\]
  \label{RES-BCYC-ASYMP}
\end{proposition}
\begin{lemma}
  Under Assumption \ref{REG-VAR-ASSUMP}, we have that
  \[ \varlimsup_{b \rightarrow \infty}\sup_{n \geq 1} \exp \left(
    n\Lambda_b (\theta_b) \right) \leq 1.\]
  \label{BCYC-NORM-TERM}
\end{lemma}
Let $\Pr_\theta(\cdot)$ and $\E_\theta[\cdot]$ denote the probability
measure and the corresponding expectation operator when the increments
$X_n$ are drawn independent from $F_\theta(\cdot).$ Since $S_{\tau_b}
> b,$ it follows from the definition of $\theta_b$ and
\eqref{RES-EST-BCYC} that
\begin{align*}
  \E_\theta \left[ \Zr ^2(b) \right] = \E \left[ \Zr (b) \right] &\leq
  \E \left[ \exp \left( -\theta_b b \right) \exp \left( \tau_b
      \Lambda_b(\theta_b)\right) ; \tau_b < \tau, \max_{k \leq \tau_b}
    X_k > b \right]\\
  &\leq b \bar{F}(b) \sup_{n \geq 1} \exp \left( n\Lambda_b (\theta_b)
  \right) \Pr \left\{ \tau_b < \tau, \max_{k \leq \tau_b}
    X_k > b \right\}\\
  &= O \left( \bar{F}^2(b)\right), \text{ as } b \nearrow \infty,
\end{align*}
because of Lemma \ref{BCYC-NORM-TERM} and Proposition
\ref{RES-BCYC-ASYMP}. Then due to \eqref{BCYC-ASYMP}, it is immediate
that
\begin{align}
  \frac{\E_\theta \left[ \Zr^2 (b) \right]}{\Pr \left\{ \tau_b < \tau
    \right\}^2} = O(1), \text{ as } b \nearrow \infty.
\label{VAR-RES-BCYC}
\end{align}
\begin{theorem}
  If the realizations of the estimators $\Zd (b)$ and $\Zr (b)$ are
  generated respectively from the measures $\Pr(\cdot)$ and
  $\Pr_{\theta} (\cdot),$ and if we let
  \[ Z(b) := \Zd(b) + \Zr(b),\] then under Assumption
  \ref{REG-VAR-ASSUMP}, the family of estimators $(Z(b): b >0)$ are
  strongly efficient for the estimation of $\Pr\{\tau_b < \tau\},$ as
  $b \nearrow \infty;$ that is,
  \[ \frac{\textnormal{Var} \left[Z(b)\right]}{\Pr\{\tau_b < \tau\}^2}
  = O(1), \text{ as } b \nearrow \infty.\]
\label{EFF-BCYC}
\end{theorem}
\begin{proof}
  Since $\Zd (b)$ and $\Zr (b)$ are generated independently,
  \begin{align*}
    \textnormal{Var} \left[ Z (b)\right] = \textnormal{Var} \left[ \Zd
      (b)\right] + \textnormal{Var} \left[ \Zr (b)\right].
  \end{align*}
  This observation, together with \eqref{VAR-DOM-BCYC} and
  \eqref{VAR-RES-BCYC} proves the claim.
\end{proof}

\section{Numerical Experiments}
\label{SEC-NUM-EG}
In this section, we present the results of numerical simulation
experiments performed on examples previously considered in literature,
and compare them with the performance of our algorithms.

\subsection{Example 1 - Estimation of $\Pr\{S_n > b\}$}
Take $X = \Lambda R,$ where $\Pr \{ \Lambda > x\} = 1 \wedge x^{-4}, R
\sim \text{Laplace}(1),$ and $\Lambda$ is independent of $R.$ We use
$N = 10,000$ simulation runs to estimate $\Pr \{ S_n > n \}$ for $n =
100, 500 \text{ and } 1000.$ In Table \ref{NUM-RES-1}, we compare the
numerical estimates obtained by our simulation procedure with the true
values of $\Pr \{ S_n > n \}$ evaluated in \cite{MR2488534} via
inverse transform techniques; further, a comparison of performance of
our methodology with Algorithms 1 and 2 in \cite{MR2488534} (referred
to as BL1 and BL2) has also been presented. From the columns CV, CV of
BL1, and CV of BL2, it can be inferred that our state-independent
simulation procedures yield estimators with substantially lower
coefficient of variation throughout the range of values
considered. The state-dependent algorithms in comparison have been
proven to be strongly efficient. The numerical performance of our
algorithms in Table \ref{NUM-RES-1} just reflects the vanishing
relative error of the estimators (a notion stronger than strong
efficiency), which has been verified in Theorem \ref{EFF-LD}.
\begin{table} [htb!]
  \centering
 \caption{Numerical result for Example 1 - here Std. error denotes the
   standard deviation of the estimator of $\Pr \{ S_n > n\}$ based on
   10,000 simulation runs;
   CV denotes the empirically observed coefficient of variation}
 \begin{tabular}{l c p{2.0cm} c c p{1cm} p{1cm}}
   \hline
   n    &$\Pr \{ S_n > n\}$ &Estimate $(\hat{z})$ for $\Pr\{ S_n> n\}$&  Std. error & CV of $\hat{z}$  & CV of BL1& CV of BL2\\ \hline
   100  & 2.21$\times 10^{-5}$ & 2.17$\times 10^{-5}$ & 4.31$\times 10^{-7}$ & 1.97 & 10.3 & 4.7 \\
   500  & 1.04$\times 10^{-7}$ & 1.05$\times 10^{-7}$ & 6.91$\times 10^{-10}$ & 0.66 & 1.0 & 4.1\\
   1000 & 1.25$\times 10^{-8}$ & 1.29$\times 10^{-8}$ & 6.91$\times 10^{-11}$ & 0.53 & 1.1 & 3.8\\\hline
 \end{tabular}
 \label{NUM-RES-1}
 \end{table}

\subsection{Example 2 - Estimation of $\Pr\{ \tau_b < \infty\}$}
To facilitate comparison with existing methods, we use the following
example from \cite{MR2434174}: Consider an M/G/1 queue with traffic
intensity $\rho = 0.5$ and Pareto service times having tail $\Pr \{ V
> t\} = (1+t)^{-2.5}.$ The aim is to estimate the probability that
this queue develops a waiting time $b$ in stationarity by equivalently
estimating the level crossing probabilities $\Pr \{ \tau_b < \infty\}$
of the associated negative drift random walk. For this example, we use
the simulation procedures discussed in Section \ref{SEC-SIM-METH} and
compare the results with that of the existing algorithms in literature
in Table \ref{NUM-RES-2}. While Algorithms AK (in \cite{AK06}) and DLW
(in \cite{Dupuis:2007:ISS:1243991.1243995}) restrict the arrivals to
be Poisson, the schemes BGL, BG and BL referring to the algorithms,
respectively, in \cite{BGL2007, MR2434174} and \cite{Blanchet20122994}
do not impose any such restriction.

In our implementation, $r$ has been chosen to be 2 to keep the
expected termination time low, as suggested by Theorem \ref{GR2-RT}.
The results reported in Table \ref{NUM-RES-2} correspond to the
simulation estimates of $\Pr \{ \tau_b < \infty \}$ for values of
$b=10^2,10^3 \text{ and } 10^4$ using $N=10,000$ simulation runs. From
Table \ref{NUM-RES-2}, it can be inferred that the error offered by
the estimates of our simpler state-independent procedure is much
smaller when compared with other existing algorithms. Table
\ref{R-COMPAR} gives a comparison of coefficient of variation of the
estimators empirically observed for different values of $r,$ and a
fixed $b=10^3.$ It can be seen from Table \ref{R-COMPAR} as well that
choosing $r=2$ helps in keeping the relative error low.


\begin{table}[htb!]
  \centering
  \caption{Numerical result for Example 2 - here Std. error denotes the
    standard deviation of the estimator of $\Pr\{ \tau_b <
    \infty\}$ based on 10,000 simulation runs;
    CV denotes the empirically observed coefficient of variation}
  \begin{tabular}{l p{0.5cm} p{3cm}  p{3cm} p{3cm} }
    \hline
    Estimation & & & &\\
    Std. error& &$b = 10^2$& $b = 10^3$& $b = 10^4$\\
    CV& & & & \\
    \hline
       & & $9.75 \times 10^{-4}$& $3.15 \times 10^{-5}$& $9.98 \times 10^{-7}$\\
    Proposed& & $4.11 \times 10^{-6}$& $7.89 \times 10^{-8}$& $1.39 \times 10^{-9}$\\
     method& & 0.42& 0.25& 0.14\\
         & & $1.20 \times 10^{-3}$& $3.15 \times 10^{-5}$& $9.98 \times 10^{-7}$\\
    AK   & & $1.48 \times 10^{-5}$& $2.19 \times 10^{-7}$& $6.95 \times 10^{-9}$\\
         & & 1.23& 0.70& 0.70\\
         & & $1.05 \times 10^{-3}$& $3.16 \times 10^{-5}$& $9.91 \times 10^{-7}$\\
    DLW  & & $5.20 \times 10^{-6}$& $1.69 \times 10^{-7}$& $2.99 \times 10^{-9}$\\
         & & 0.50& 0.53& 0.30\\
         & & $1.02 \times 10^{-3}$& $3.17 \times 10^{-5}$& $1.13 \times 10^{-6}$\\
    BGL  & & $3.84 \times 10^{-5}$& $1.60 \times 10^{-6}$& $7.28 \times 10^{-8}$\\
         & & 3.76& 5.05& 6.44\\
         & & $1.08 \times 10^{-3}$& $3.15 \times 10^{-5}$& $9.98 \times 10^{-7}$\\
    BG   & & $5.97 \times 10^{-6}$& $9.73 \times 10^{-8}$& $2.07 \times 10^{-9}$\\
         & & 0.55& 0.31& 0.21\\
         & & $1.05 \times 10^{-3}$& $3.18 \times 10^{-5}$& $9.88 \times 10^{-7}$\\
    BL   & & $3.76 \times 10^{-5}$& $2.60 \times 10^{-7}$& $8.19 \times 10^{-9}$\\
         & & 3.58& 0.82& 0.83\\
    \hline
  \end{tabular}
  \label{NUM-RES-2}
 \end{table}
\begin{table}[htb!]
  \centering
 \caption{Comparison of relative errors for different choices of $r$ in
   Example 2 with $b=1000;$ here Std. error denotes the
   standard deviation of the estimator of $\Pr\{ \tau_b <
   \infty\}$ based on 10,000 simulation runs;
   CV denotes the empirically observed coefficient of variation}
  \begin{tabular}{l | c c c}
    \hline
    $r$  & Estimate            & Std. error          & CV  \\ \hline
    2  & 3.15$\times 10^{-5}$ & 7.89$\times 10^{-8}$ & 0.25\\
    10 & 3.16$\times 10^{-5}$ & 1.03$\times 10^{-7}$ & 0.33\\
    100& 3.16$\times 10^{-5}$ & 1.55$\times 10^{-7}$ & 0.49\\ \hline
 \end{tabular}
 \label{R-COMPAR}
 \end{table}

\section{Conclusion}
\label{SEC-CONC}

In this paper we revisited the problem of efficient simulation of
commonly encountered rare event probabilities associated with random
walks having regularly varying heavy-tailed increments.  These
comprised the large deviations probability of a random walk exceeding
large values as well as level crossing probabilities corresponding to
negative-drift random walks. In the existing literature there are
results that suggest that state-independent methods for such
probabilities are difficult to design. Significant research over the
last few years has resulted in sophisticated state-dependent
importance sampling techniques for estimating these probabilities. Our
key contribution has been to challenge this view by showing that
simple state-independent importance sampling methods, that are at
least as efficient as the existing state-dependent methods, can indeed
be devised to estimate these probabilities.

Our approach relied on partitioning the rare event of interest into
elementary events that are amenable to straight forward
state-independent importance sampling methods.  We expect that this
approach will generalize to more complex, multi-dimensional problems,
and for similar problems involving Weibull-type sub-exponential tail
distributions.

\appendix
\section{Proofs of certain probability estimates}
In this section we present proofs of Propositions \ref{LD-RES-PROB},
\ref{AK-ASYMP}, \ref{RES-BCYC-ASYMP} and Lemma
\ref{RES-TERM-ASYMP}. These asymptotic results on certain
probabilities of interest have been useful in efficiency analysis of
our algorithms. Some of these involve error estimates that have not
been studied in the literature, and are interesting in their own
right.

\paragraph{Proof of Proposition \ref{LD-RES-PROB}}
Let $M_n := \max\{X_1,\ldots,X_n\}$ for $n \geq 1.$ We first obtain a
lower bound for $\Pr \{ S_n > b, X_n > b, M_{n-1} \leq b\}:$ 
\begin{align}
  \Pr \{ S_n > b, X_n > b, M_{n-1} \leq b\} &\geq \Pr \left\{ S_{n-1}
    >
    -b^\gamma, M_{n-1} \leq b, X_n > b + b^\gamma \right\} \nonumber\\
  &= \Pr \left\{S_{n-1} > -b^\gamma, M_{n-1} \leq b \right\}
  \bar{F}\left(b+b^\gamma\right) \label{LB-LEM1-INTER}
\end{align}
for some $\gamma < 1$ to be chosen later in the proof. Due to
\eqref{MAX-BONF-INEQ},
\[\Pr \{M_{n-1} > b\} \sim (n-1)\bar{F}(b) \searrow 0\]
uniformly for all $b > n^{\beta+\epsilon},$ as $n \nearrow \infty.$
Here recall that $\beta := (\alpha \wedge 2)^{-1}.$ Similarly for
$\gamma > \beta/(\beta+\epsilon),$ because of the convergence of
$S_n/n^\beta$ to the stable distribution, we have $\Pr \{S_{n-1} <
-b^\gamma\} \searrow 0, $ uniformly for all $b > n^{\beta+\epsilon},$
as $n \nearrow \infty.$ Therefore, it follows from union bound that,
\[ \Pr \left\{S_{n-1} \geq -b^\gamma, M_{n-1} \leq b \right\} \geq 1 -
o(1),\] uniformly for all $b > n^{\beta+\epsilon},$ as $n \nearrow
\infty.$ Since $\gamma < 1,$ 
\[\frac{\bar{F}(b+b^\gamma)}{\bar{F}(b)} \geq 1-o(1)\] because of 
\eqref{LONG-TAIL-EXT}. Combining these observations with
\eqref{LB-LEM1-INTER}, it follows that
\begin{align}
  \Pr \{ S_n > b, X_n > b, M_{n-1} \leq b \} &\geq (1-o(1))\bar{F}(b)
  \label{LB-LEM1-INTER}
\end{align}
uniformly for all $b > n^{\beta+\epsilon},$ as $n \nearrow \infty.$\\
Since $\Pr(\Ar(n,b)) = \Pr\{ S_n > b\} - \Pr \{S_n > b, M_n > b \},$
\begin{align*}
  \Pr(\Ar(n,b)) &\leq \Pr\{ S_n > b\} - \sum_{j=1}^n \Pr \left\{ S_n >
    b, X_j > b, \max_{i \neq j, i \leq n} X_i \leq b\right\}\\
  & = \Pr\{ S_n > b\} - n\Pr \{ S_n > b, X_n > b, M_{n-1} \leq b \}\\
  &\leq (1+o(1))n\bar{F}(b)-(1-o(1))n\bar{F}(b) = o \left(n\bar{F}(b)
  \right),
\end{align*}
where the last inequality follows from \eqref{LD-PROBS} and
\eqref{LB-LEM1-INTER}.\hfill\(\Box\)

\paragraph{Proof of Proposition \ref{AK-ASYMP}}
The upper bound follows simply by applying union bound as below:
\begin{align}
  \Pr \{\evk, A_k\} &\leq \Pr
  \left\{\bigcup_{j=\nkm+1}^{n_k}\left\{X_j > b+j\mu\right\}
  \right\} \nonumber\\
  &\leq \sum_{j=\nkm+1}^{n_k} \bar{F}(b+j\mu) = q_k(b).
  \label{AK-ASYMP-UB}
\end{align}
For obtaining a lower bound, see that \[\Pr\{\evk, A_k\} =
\sum_{j=\nkm + 1}^{n_k}\Pr\{\tau_b=j,A_k\}\] is bounded from below by
\begin{align*}
  \sum_{j=\nkm + 1}^{n_k}\Pr \left\{\tau_b=j, S_i > -(b+i \mu)^\gamma
    \text{ for all } i < j, X_j > b + j\mu + (b+j \mu)^\gamma
  \right\},
\end{align*}
for some $\gamma < 1$ to be chosen later in this proof.  $\text{Let }
M_n := \max_{k \leq n}(S_k - k\mu) \text{ and } M := \sup_k (S_k -
k\mu).$ Then $\Pr\{\evk, A_k\}$ is lower bounded by
\begin{align}
  &\sum_{j=\nkm + 1}^{n_k}\Pr \left\{M_{j-1} \leq b, S_i > -(b+i
    \mu)^\gamma \text{ for all } i < j, X_j > b + j\mu + (b+j
    \mu)^\gamma \right\}\nonumber\\
  &\quad\hspace{-1pt}= \sum_{j=\nkm + 1}^{n_k}\Pr \left\{M_{j-1} \leq
    b, S_i > -(b+i \mu)^\gamma \text{ for all } i < j \right\}
  \bar{F}\left(b + j\mu + (b+j\mu)^\gamma\right)\nonumber\\
  &\quad\hspace{-1pt}\geq \Pr \left\{M \leq b, \min_{i < n_k} S_i >
    -(b+\nkm \mu)^\gamma \right\} \sum_{j=\nkm + 1}^{n_k}
  \bar{F}\left(b + j\mu + (b+j\mu)^\gamma\right) \label{AK-LB-INT1}
\end{align}
From \eqref{ASYMP} we have that $\Pr\{ M > b \} \sim \mu^{-1}
\bar{F}_I(b)$ as $b \nearrow \infty.$ Recall that $\beta = (\alpha
\wedge 2)^{-1}.$ If $\gamma > \beta,$ then under the lighter left tail
assumption formally stated in Assumption
\ref{LEFT-TAIL-LIGHTER-ASSUMP},
\[ \Pr\left\{ \min_{i < n_k} S_i < -(b+\nkm \mu)^\gamma \right\} =
O\left(n_k(b+\nkm \mu)^{-\frac{\gamma}{\beta} + o(1)} \right), \text{
  as } b \nearrow \infty.\] This follows from the well-known large
deviation asymptotic that
\[\Pr \left\{\max_{i \leq n} S_i > x \right\} \sim n\bar{F}(x)\]
uniformly for $x > n^{\beta+\epsilon};$ this can be found, for
example, in Theorem 2.2 of \cite{borovkovboxma} and Theorem 5 of
\cite{doi:10.1137/S0040585X97978877}. Therefore, by union bound,
\begin{align}
  \Pr &\left\{M \leq b,\min_{i < n_k} S_i > -(b+\nkm
    \mu)^\gamma \right\} \nonumber\\
  &\hspace{20pt}\geq 1 - \Pr\{M > b\} - \Pr\left\{ \min_{i < n_k} S_i
    < -(b+\nkm \mu)^\gamma \right\} \nonumber\\
  &\hspace{20pt}\geq 1 - \bar{F}_I(b)(1-o(1)) - O\left(n_k(b+\nkm
    \mu)^{-\frac{\gamma}{\beta} + o(1)} \right), \label{AK-LB-INT2}
\end{align}
as $b \nearrow \infty.$ Further because of \eqref{LONG-TAIL-EXT},
\begin{align*}
  \sum_{j=\nkm + 1}^{n_k} \bar{F}\left(b + j\mu +
    (b+j\mu)^\gamma\right) &\geq \sum_{j=\nkm + 1}^{n_k} \left(1 +
    \frac{(b+j\mu)^{\gamma}}{b+j\mu} \right)^{-\alpha+o(1)}
  \bar{F}\left(b + j\mu\right)\\
  &\geq \left( 1-\frac{c}{(b+\nkm \mu)^{1-\gamma}}\right) \sum_{j=\nkm
    + 1}^{n_k} \bar{F}\left(b + j\mu\right)
\end{align*}
for some positive constant $c.$ If we choose $\gamma =
(\alpha+1)/2\alpha,$ then
\[ \sum_{j=\nkm + 1}^{n_k} \bar{F}\left(b + j\mu +
  (b+j\mu)^\gamma\right) \geq \left( 1-\frac{c}{(b+\nkm
    \mu)^{\frac{\alpha-1}{2\alpha}}}\right)q_k(b).\] Combining this
with \eqref{AK-LB-INT1} and \eqref{AK-LB-INT2}, we see that
\begin{align}
  \Pr \{\evk,A_k\} \geq \left( 1-\frac{c+o(1)}{
      b^{\frac{\alpha-1}{2\alpha}}} \right)q_k(b),
\label{AK-LB-WITH-ERR-TERM}
\end{align}
as $b \nearrow \infty.$ Along with \eqref{AK-ASYMP-UB}, we have that 
\[ \sup_k \left|\frac{\Pr\{\evk,A_k\}}{q_k(b)}-1\right| = O\left(
  b^{-\frac{\alpha-1}{2\alpha}}\right), \text{ as } b \nearrow
\infty. \]
\hfill\(\Box\) 

\paragraph{Proof of Lemma \ref{RES-TERM-ASYMP}}
Recall that $n_k = r \nkm$ for some constant $r.$ Therefore, for any
$k \geq 1,$
\begin{align*}
  1 \leq
  \frac{\sum_{j=1}^{n_k}\bar{F}(b+j\mu)}{\sum_{j=1}^{\nkm}\bar{F}(b+j\mu)}
  = 1 +
  \frac{\sum_{j=\nkm+1}^{n_k}\bar{F}(b+j\mu)}{\sum_{j=1}^{\nkm}\bar{F}(b+j\mu)}
  \leq 1+\frac{\nkm \bar{F}(b+\nkm \mu)}{\nkm \bar{F}(b+\nkm \mu)} = 2.
\end{align*}
When $\text{Var}[X]< \infty,$ see from
\eqref{FINITE-HOR-ASYMP-FIN-VAR} and Proposition \ref{AK-ASYMP} that
$\Pr\{\evk,\bar{A}_k\}$ equals
\begin{align}
  &\Pr\{\tau_b \leq n_k\} - \Pr\{\tau_b \leq \nkm \} -
  \Pr\{\evk,A_k\} \nonumber\\
  &=
  \sum_{j=1}^{n_k}\bar{F}(b+j\mu)\left(1+O\left(\frac{1}{b}\right)\right)
  + o\left(\sqrt{n_k \wedge b} \ \bar{F}(b)\right) \nonumber\\
  &\hspace{40pt}-\sum_{j=1}^{\nkm}\bar{F}(b+j\mu)\left(1+O\left(\frac{1}{b}\right)\right)-\left(1-O\left(b^{-\frac{\alpha-1}{2\alpha}}
    \right)\right) \sum_{j=\nkm+1}^{n_k}\bar{F}(b+j\mu), \nonumber\\
  &= \sum_{j=1}^{n_k}\bar{F}(b+j\mu)\left(
    O\left(b^{-\frac{\alpha-1}{2\alpha}}\right)\right) +
  o\left(\sqrt{n_k \wedge b} \ \bar{F}(b)\right)
\label{FIN-VAR-LEM-PF}
\end{align}
as $b \nearrow \infty.$ Similarly when $\text{Var}[X] = \infty,$ for
$k$ such that $n_k\bar{F}(b) = o(1),$ see from
\eqref{FINITE-HOR-ASYMP-INF-VAR} and Proposition \ref{AK-ASYMP} that
$\Pr\{\evk,\bar{A}_k\}$ equals
\begin{align}
  \sum_{j=1}^{n_k}\bar{F}(b+j\mu)\left(1+O\left(\frac{n_k^{\frac{1}{\alpha}+\epsilon}}{b}\right)\right)
  - \sum_{j=1}^{\nkm}\bar{F}(b+j\mu)\left(1+O\left(\frac{\nkm
        ^{\frac{1}{\alpha}+\epsilon}}{b}\right)\right) \nonumber\\
  - \left(1-O\left(b^{-\frac{\alpha-1}{2\alpha}} \right)\right)
  \sum_{j=\nkm +1}^{n_k}\bar{F}(b+j\mu)\nonumber\\
  =\sum_{j=1}^{n_k}\bar{F}(b+j\mu)\left(O\left(\frac{n_k^{\frac{1}{\alpha}+\epsilon}}{b}\right)
    + O\left(b^{-\frac{\alpha-1}{2\alpha}} \right) \right)
  \label{INF-VAR-LEM-PF}
\end{align}
for every $\epsilon > 0.$ Since $n_k = r\nkm$ for some constant $r,$
it follows from \eqref{LONG-TAIL-EXT} that for small enough $\epsilon$
and suitably chosen $\eta > 1,$
\[ \sup_{k:n_k < b^\eta} \frac{b^{-\frac{\alpha-1}
    {2\alpha}}\sum_{j=1}^{n_k}\bar{F}(b+j\mu)}{q_k(b)} \leq
\sup_{k:n_k < b^\eta} \frac{b^{-\frac{\alpha-1}
    {2\alpha}}n_k\bar{F}(b)}{\nkm\bar{F}(b+n_k\mu)} = o(1), \]
\[ \sup_{k:n_k < b^\eta}\frac{n_k^{\frac{1}{\alpha}+\epsilon}}{b}
\frac{\sum_{j=1}^{n_k}\bar{F}(b+j\mu)}{q_k(b)} \leq \sup_{k:n_k <
  b^\eta}
\frac{n_k^{\frac{1}{\alpha}+\epsilon}}{b}\frac{n_k\bar{F}(b)}{\nkm\bar{F}(b+n_k\mu)}
= o(1), \text{ and }\]
\[\sup_{k:n_k < b^\eta} \frac{\sqrt{n_k\wedge b} \ 
  \bar{F}(b)}{\sum_{j=1}^{n_k}\bar{F}(b+j\mu)} \leq \sup_{k:n_k <
  b^\eta} \frac{\sqrt{n_k\wedge b} \ \bar{F}(b)}{n_k\bar{F}(b+n_k\mu)}
= o(1),\] as $b \nearrow \infty.$ Therefore from
\eqref{FIN-VAR-LEM-PF} and \eqref{INF-VAR-LEM-PF}, for some $\eta >
1,$
\begin{align}
  \sup_{k:n_k < b^\eta} \frac{\Pr\{\evk,\bar{A}_k\}}{q_k(b)} = o(1),
  \text{ as } b \nearrow \infty.
  \label{FIRST-PART}
\end{align}
For $k$ such that $n_k > b^\eta,$ we obtain a loose bound that
suffices for our purposes:
\begin{align}
  \Pr \{ \evk \} &= \Pr \left\{ \evk , S_{\nkm} > \frac{b+\nkm
      \mu}{2}\right\} \nonumber\\
  & \hspace{20pt}+ \Pr \left\{ \evk , S_{\nkm} \leq \frac{b+\nkm
      \mu}{2}\right\} \nonumber\\
  &\leq \Pr \left\{S_{\nkm} > \frac{b+\nkm \mu}{2}\right\} + \Pr
  \left\{  \tau_{_\frac{b+\nkm \mu}{2}} < \infty \right\} \nonumber\\
  &\leq (1+\epsilon) \nkm\bar{F}\left(\frac{b+\nkm \mu}{2}\right) +
  \frac{(1+\epsilon)}{\mu} \bar{F}_I\left(\frac{b+\nkm \mu}{2}\right),
  \label{INTER-EQ}
\end{align}
for all $k,b$ large enough. While the final inequality is due to
the asymptotics \eqref{LD-PROBS} and \eqref{ASYMP}, the second term in
the penultimate step follows by observing that whenever the event $\{
\evk , S_{\nkm} \leq (b+\nkm \mu)/2\}$ happens, the random walk $(Z_n:
n \geq 0)$ defined by
\[Z_n := S_{n+\nkm}-S_{\nkm} - n\mu\] crosses the level $(b + \nkm
\mu)/2$ at some finite $n \leq n_k-\nkm.$ Here recall that
$\tau_x:=\inf\{k \geq 1: S_k > x + k\mu \}.$\\
Further, since $n_k = r\nkm$ for some constant $r,$ from
\eqref{KARAMATA} and \eqref{LONG-TAIL-EXT}, we have that
\[ \sup_{k:n_k > b^\eta} \frac{\nkm\bar{F}\left(\frac{b+\nkm
      \mu}{2}\right)}{q_k(b)} < \infty \text{ and } \sup_{k:n_k >
  b^\eta} \frac{\bar{F}_I\left(\frac{b+\nkm \mu}{2}\right)}{q_k(b)} <
\infty.\] Therefore from \eqref{INTER-EQ},
\begin{align*}
  \sup_{k:n_k > b^\eta} \frac{\Pr\{\evk,\bar{A}_k\}}{q_k(b)} < \infty,
\end{align*}
which along with \eqref{FIRST-PART} establishes the
claim. \hfill\(\Box\) 

\paragraph{Proof of Proposition \ref{RES-BCYC-ASYMP}} 
Consider
\begin{align*}
  P_1 &:= \Pr \left\{ \tau_b < \tau, \max_{k \leq \tau_b} X_k \leq b,
    S_{\tau_b-1} < \frac{b}{2}\right\}\\
  &\leq \Pr \left\{ S_n \in \left(0,\frac{b}{2}\right), S_{n+1} > b,
    X_{n+1} < b \text{ for some }  n < \tau \right\}\\
  &\leq \E \left[ \sum_{n=0}^{\tau-1} \mathbb{I}\left( S_n \in
      \left(0,\frac{b}{2}\right), S_{n+1} > b, X_{n+1} < b
    \right)\right]\\
  &\leq \E \left[ \sum_{n=0}^{\tau-1} \mathbb{I}\left( S_n \in
      \left(0,\frac{b}{2}\right)\right) \Pr\left\{b-S_n < X < b
    \right\}\right].
\end{align*}
Let $\pi(B) = \Pr\{ \sup_n S_n \in B \}.$ Then according to the
regenerative ratio representation,
\[ \frac{1}{\E \tau}\E \left[ \sum_{n=0}^{\tau-1} \mathbb{I}\left( S_n
    \in \left(0,\frac{b}{2}\right)\right) \Pr\left\{b-S_n < X < b
  \right\}\right] = \int_0^{\frac{b}{2}} \Pr\left\{b-u < X < b
\right\} \pi(du).\] Therefore,
\begin{align*}
  P_1 &\leq \E \tau \int_0^{\frac{b}{2}} \left(\bar{F}(b-u) -
    \bar{F}(b)\right) \pi(du)\\
  &= \E \tau \bar{F}(b)\int_0^{\frac{b}{2}}
  \left(\frac{\bar{F}(b-u)}{\bar{F}(b)} - 1 \right) \pi(du). 
\end{align*}
From Potter's bounds (see \eqref{LONG-TAIL-EXT}), we have that for all
$u < \frac{b}{2},$
\[ \frac{\bar{F}(b-u)}{\bar{F}(b)} \leq
\left(1-\frac{u}{b}\right)^{-\alpha-\delta} \leq 1 + (\alpha + \delta)
2^{\alpha+\delta+1} \frac{u}{b},\] for any $\delta > 0$ and all $b$
large enough. The last inequality follows from Taylor's theorem. Hence
\begin{align*}
  P_1 \leq (\alpha + \delta) 2^{\alpha+\delta+1} \E \tau
  \frac{\bar{F}(b)}{b}\int_0^{\frac{b}{2}}u\pi(du).
\end{align*}
Recall that $\pi((x,\infty)) \sim \int_x^\infty\bar{F}(u)du.$ Since
$\alpha > 2,$ $\int_0^\infty u \pi(du) < \infty.$ Therefore,
\begin{align}
  \label{P1}
  P_1 = O \left( \frac{\bar{F}(b)}{b}\right), \text{ as } b \nearrow
  \infty. 
\end{align}
Now consider the complementary event $\left\{ \tau_b < \tau, \max_{k
    \leq \tau_b} X_k \leq b, S_{\tau_b-1} > b/2 \right\}:$
\begin{align*}
  P_2 &:= \Pr \left\{ \tau_b < \tau, \max_{k \leq \tau_b} X_k \leq b,
    S_{\tau_b-1} > \frac{b}{2}\right\}\\
  &= \sum_{n=1}^\infty \Pr \left\{ \tau > n, \max_{k \leq n}
    X_k \leq b,  S_{n-1} > \frac{b}{2}, \tau_b = n\right\}\\
  &= \sum_{n=1}^\infty \Pr \left\{ \tau > n, \max_{k \leq n} X_k \leq
    b, S_{n-1} > \frac{b}{2}, M_{n-1} \leq b, S_n >
    b\right\}\\
  &= \sum_{n=1}^\infty \int_{\frac{b}{2}}^b \Pr \left\{ \tau > n,
    \max_{k \leq n} X_k \leq b, S_{n-1} \in dy, M_{n-1} \leq b, S_n >
    b \right\}\\
  &\leq \sum_{n=1}^\infty \int_{\frac{b}{2}}^b \Pr \left\{ \tau > n,
    S_{n-1} \in dy, M_{n-1} \leq b \right\} F((b-y,b]),\\
\end{align*}
where the notation $F((x,y])$ stands for $\Pr\{x < X \leq y\}.$
Consider the taboo renewal function $H_x(\cdot)$ defined below:
\[ H_x(B) := \sum_{n=0}^\infty \Pr \left\{ \tau > n, M_n \leq x, S_n
  \in B \right\}. \] 
Then it is immediate that
\[ P_2 \leq \int_{\frac{b}{2}}^b H_b(dy) F((b-y,b]).\] From Theorem 2
of \cite{denisov2007}, given $\epsilon > 0,$ we have a $y_0$ large
enough such that, for all $x$ and $y$ with $y \in (y_0,x-y_0),$
\[ (1-\epsilon) \frac{\E \tau}{\mu} F((y,x])dy \leq H_x((y,y+dy)) \leq
(1+\epsilon) \frac{\E \tau}{\mu} F((y,x])dy.\] Therefore, for a fixed
$\epsilon > 0,$ we have
\[ H_{b+c}((y,y+dy)) \leq (1+\epsilon) \frac{\E \tau}{\mu}
F((y,b+c])dy \] in the interval $(b/2,b),$ for some constant $c$ and
all $b$ large enough. Since $H_b(\cdot) \leq H_{b+c}(\cdot),$
\begin{align*}
  P_2 &\leq (1+\epsilon) \frac{\E \tau}{\mu} \int_{\frac{b}{2}}^b
  F((y,b+c]) F((b-y,b])dy\\
  &= (1+\epsilon) \frac{\E \tau}{\mu} \int_0^{\frac{b}{2}}
  F((b-y,b+c]) F((y,b])dy\\
  &\leq (1+\epsilon) \frac{\E \tau}{\mu} \left(\bar{F}(b)
    \int_{0}^{\frac{b}{2}} \frac{F((b-y,b])}{\bar{F}(b)} F((y,b])dy +
    \int_{0}^{\frac{b}{2}}F((b,b+c])F((y,b])dy
  \right)\\
\end{align*}
For a fixed $\delta > 0,$ it follows from \eqref{LONG-TAIL-EXT} that
\begin{align*}
  \frac{F((b-y,b])}{\bar{F}(b)} &= \bar{F}(b) \left(
    \frac{\bar{F}(b-y)}{\bar{F}(b)} - 1\right)\\
  &\leq \bar{F}(b) \left( \left(1-\frac{y}{b}\right)^{-\alpha-\delta}
    -1\right)\\
  &\leq (\alpha+\delta)2^{\alpha+\delta+1} \frac{y}{b}\bar{F}(b)
\end{align*}
for all $y < b/2$ and $b$ large enough. The last inequality is a
consequence of Taylor's theorem. Then,
\begin{align}
  P_2 &\leq (1+\epsilon) \frac{\E \tau}{\mu} \left(
    (\alpha+\delta)2^{\alpha+\delta+1} \frac{\bar{F}(b)}{b}
    \int_{0}^{\frac{b}{2}} y\bar{F}(y) dy + F((b,b+c])
    \int_{0}^{\frac{b}{2}}F((y,b])dy \right).
  \label{P2-INTER}
\end{align}
Since $\bar{F}(\cdot)$ is regularly varying, given $\gamma > 0$ it
follows from \eqref{LONG-TAIL-EXT} that for all $b$ large enough,
\begin{align*}
  F((b,b+c]) &=
  \bar{F}(b)\left(1-\frac{\bar{F}(b+c)}{\bar{F}(b)}\right)\\
  &\leq \bar{F}(b)\left(1-\left(1+\frac{c}{b}\right)^{-\alpha-\gamma}\right)\\
  &\leq \bar{F}(b) \left((\alpha+\gamma)\frac{c}{b}\right).
\end{align*}
Therefore
\[F((b,b+c]) = O\left( \frac{\bar{F}(b)}{b}\right), \text{ as } b
\nearrow \infty.\] Further, when $\alpha > 2,$ 
\[\int_{0}^\infty \bar{F}(u)du < \infty \text{ and } \int_{0}^\infty
u\bar{F}(u)du < \infty.\] Using these in \eqref{P2-INTER}, we obtain
that for tails with regularly varying index $\alpha > 2,$
\begin{align}
  P_2 = O\left(\frac{\bar{F}(b)}{b} \right),
  \label{P2}
\end{align}
as $b \nearrow \infty.$ Therefore, from \eqref{P1} and \eqref{P2}, we
obtain
\[ \Pr \left\{ \tau_b < \tau, \max_{k \leq \tau_b} X_k < b \right\} =
P_1 + P_2 = O\left(\frac{\bar{F}(b)}{b} \right), \] as $b \nearrow
\infty.$ This proves the claim. \hfill$\Box$

\section{Proofs of other lemmas}
Here we present proofs of Lemmas \ref{LEM-LD}, \ref{NORM-TERM},
\ref{OTHER-TERM} and \ref{BCYC-NORM-TERM}. To prove Lemmas
\ref{LEM-LD}, \ref{NORM-TERM} and \ref{BCYC-NORM-TERM}, we need Lemmas
\ref{Rthx} and \ref{theta-rate}, which are stated and proved
below. The proof of Lemma \ref{Rthx} follows the lines of Theorem
4.1.2 of \cite{MR2424161}, where bounds for similar integrals have
been derived.
\begin{lemma}
  For any pair of sequences $\{x_n\},\{ \fin \}$ satisfying $x_n
  \nearrow \infty$ and $\phi _n x_n \nearrow \infty,$ the integral,
  \begin{align*}
    \int_{-\infty}^{x_n} e^{\fin x}F(dx) \leq 1 + c \fin^\kappa + e^{2
      \alpha }\bar{F}\left( \frac{2 \alpha}{\fin} \right) + e^{\fin
      x_n} \bar{F}(x_n) (1+ o(1)),
  \end{align*}
  as $n \nearrow \infty,$ for any $1 < \kappa < \alpha \wedge 2,$ and
  some constant $c$ which does not depend on $n$ and $b.$
  \label{Rthx}
\end{lemma}
\begin{proof}
  We split the region of integration into $(-\infty, \gn]$ and $(\gn,
  x_n]$ for some constant $\gamma > 0$; the partition is such that the
  integrand stays bounded in the former region.\\
  Let $I_1 := \int_{-\infty}^{\gn} e^{\fin x} F(dx)$ and $I_2 :=
  \int_{\gn}^{x_n} e^{\fin x} F(dx).$\\
  For any $\kappa \in (1,2]$ and $y > 0,$ it is easily verified that
  \[ e^x \leq 1 + x + |x|^\kappa e^y, \quad x \in (-\infty,y].\]
  Therefore,
\begin{align}
  I_1 &\leq \int_{-\infty}^{\gn}\left(1 + \fin x + \fin^\kappa
    |x|^\kappa \exp(\fin \cdot \gn) \right) F(dx)
  \nonumber\\
  &\leq \int_{-\infty}^{\gn}F(dx) + \fin\int_{-\infty}^{\gn}xF(dx) +
  \fin ^\kappa \ega\int_{-\infty}^{\gn} |x|^\kappa F(dx)
  \nonumber\\
  &\leq \int_{-\infty}^{\infty}F(dx) + \fin
  \int_{-\infty}^{\infty}xF(dx) + \fin ^\kappa \ega
  \int_{-\infty}^{\infty} |x|^\kappa F(dx) \nonumber\\
  &= 1 + c \fin ^\kappa, \label{int-I1}
\end{align}
where $c := \ega \int_{-\infty}^{\infty} |x|^\kappa F(dx) < \infty$
because $\E |X|^\kappa < \infty;$ this follows because $\kappa <
\alpha$ and from Assumption \ref{LEFT-TAIL-LIGHTER-ASSUMP}. We have
also used $\E X = 0$ to arrive at \eqref{int-I1}.  Integrating by
parts for the second integral $I_2:$
\begin{align}
  I_2 &= -\int_{\gn}^{x_n} e^{\fin x} \bar{F}(dx) \nonumber\\
  &= e^{\fin \gn} \bar{F}\left(\frac{\gamma}{\phi_n}\right) - e^{\fin
    x_n}\bar{F}(x_n) + \fin
  \int_{\gn}^{x_n} e^{\fin x} \bar{F}(x)dx \nonumber\\
  &\leq \ega \bar{F}\left(\frac{\gamma}{\phi_n} \right) + I_2',
\label{int-I2}
\end{align}
where, $I_2' := \fin \int_{\gn}^{x_n} e^{\fin x}
\bar{F}(x) dx.$ Now the change of variable $u = \fin (x_n - x)$  results in:
\begin{align}
  I_2' &= e^{\fin x_n} \int_0^{\fin x_n - \gamma} e^{-u}\bar{F} \left(x_n -
    \frac{u}{\fin } \right) du \nonumber\\
  &= e^{\fin x_n} \bar{F}(x_n) \int_0^{\fin x_n - \gamma} e^{-u} g_n(u)
  du,
\label{int-I2'}
\end{align}
where,
\begin{equation*}
  g_n(u) := \frac{\bar{F} \left(x_n - \frac{u}{\fin }
    \right)}{\bar{F}(x_n)} = \frac{\bar{F} \left(x_n \left(1-
        \frac{u}{\fin x_n} \right) \right)}{\bar{F}(x_n)}.
\end{equation*}
Since $L(\cdot)$ is slowly varying and $\fin x_n \rightarrow \infty,$
given any $\delta > 0,$ it follows from \eqref{LONG-TAIL-EXT} that,
\begin{equation*}
(1- \delta) \left(1-\frac{u}{\fin x_n} \right) ^{-\alpha + \delta}
\leq g_n(u) \leq (1 + \delta )\left(1-\frac{u}{\fin x_n} \right) ^{-\alpha - \delta}.
\end{equation*}
for all $n$ large enough. So for any fixed $u,$ we have $g_n(u)
\rightarrow 1$ as $n \nearrow \infty.$ Now fix $\delta =
\frac{\alpha}{2}.$ Then for $n$ large enough,
\begin{equation}
  g_n(u) \leq \left( 1 + \frac{\alpha}{2} \right) \left( 1
    -\frac{u}{\fin x_n }  \right) ^{-\frac{3\alpha }{2}}.
\label{gn-bound}
\end{equation}
Let $h(u) = \left( 1 -u/\fin x_n \right) ^{-\frac{3\alpha
  }{2}}$. Since $\log h(0) = 0$ and $\frac{d}{du}\left( \log(h(u)
\right) \leq \frac{3 \alpha}{2 \gamma }$ for $0 \leq u \leq \fin x_n -
\gamma ,$ we have $h(u) \leq \exp({3 \alpha u}/{2 \gamma })$ on the
same interval. Therefore if we choose $\gamma = 2 \alpha,$ the
integrand in $I_2'$ is bounded for large enough $n$ by an integrable
function as below:
\begin{align*}
  \left| e^{-u}g_n(u)\mathbf{1}(0 \leq u \leq \fin x_n - \gamma)
  \right| &\leq \left| e^{-u} \left( 1 + \frac{\alpha }{2} \right)
    h(u)\mathbf{1}(0 \leq u \leq \fin x_n - \gamma) \right|\\
  &\leq \left( 1 + \frac{\alpha }{2} \right) e^{-u + \frac{3 \alpha
      u}{2 \gamma }} = \left( 1 + \frac{\alpha }{2} \right) e
  ^{-\frac{u}{4}}.
\end{align*}
Applying dominated convergence theorem, we get
\begin{equation*}
\int_0^{\fin x_n -
  \gamma} e^{-u} g_n(u)du \sim 1 \text{ as } n \nearrow \infty.
\end{equation*}
Since $\int_{-\infty}^{x_n} e^{\fin x}F(dx) = I_1 + I_2,$ combining
this result with \eqref{int-I1}, \eqref{int-I2} and \eqref{int-I2'},
completes the proof.
\end{proof}

\begin{lemma}
  Given any $\epsilon > 0,$ uniformly for $b > \thr,$ we have:
\begin{enumerate}[(a)]
\item $n \tnb^\kappa \searrow 0$ for some $1 < \kappa < \alpha \wedge
  2,$ and
\item $\bar{F}\left( 2 \alpha/\tnb \right) = o\left( {1}/{n}\right),$
  as $n \nearrow \infty.$
\end{enumerate}
\label{theta-rate}
\end{lemma}
\begin{proof}
  (a) We have $\bar{F}(x) = x^{-\alpha }L(x).$ Since $L(\cdot)$ is
  slowly varying, following \eqref{LONG-TAIL-EXT} we have that $L(b) =
  b^{o(1)}$ as $b \nearrow \infty.$ Further noting that $b >
  n^{\beta+\epsilon}$ helps us to write:
\begin{align*}
  n \tnb ^\kappa &= \frac{n}{b^\kappa} \log^\kappa \left(
    \frac{1}{n\bar{F}(b)} \right) \leq n^{1-\kappa (\beta+\epsilon)}
  \log^\kappa \left( \frac{b^{\alpha}}{nL(b)} \right).
\end{align*}
If we choose $\kappa \in ((\beta + \epsilon)^{-1}, \alpha)$ then
$\kappa (\beta + \epsilon) > 1$ and subsequently $n \tnb^\kappa
\searrow 0$ as $n \nearrow \infty,$ uniformly for $b > \thr.$\\\\
(b) We have $\tnb := {-\log \left(n\bar{F}(b)\right)}/{b}.$ Therefore,
  \begin{equation*}
    n\bar{F}\left( \frac{2 \alpha}{\tn} \right) =
    n\bar{F}(b)\frac{\bar{F}\left(\frac{2 \alpha
          b}{-\log(n\bar{F}(b))} \right)}{\bar{F}(b)}.
  \end{equation*}
  Since $\bar{F}(\cdot)$ is regularly varying, given any $\delta > 0,$
  it follows from \eqref{LONG-TAIL-EXT} that
\begin{align*}
  \frac{\bar{F}\left(\frac{2 \alpha b}{-\log(n\bar{F}(b))}
    \right)}{\bar{F}(b)} &\leq \left( \frac{ -\log\left( \pn \right)
    }{2 \alpha } \right)^{\alpha + \delta },
\end{align*}
for $n$ large enough.  Therefore,
  \begin{align*}
    n\bar{F}\left( \frac{2 \alpha}{\tn} \right) \leq
    n\frac{L(b)}{b^{\alpha}}\left( \frac{ -\log\left( \pn \right) }{2
        \alpha } \right)^{\alpha + \delta } = o(1),
  \end{align*}
  uniformly for $b > \thr \text{ as } n \nearrow \infty.$ Here the
  convergence to 0 is justified because $\alpha > 1$ and $b > \thr.$
  \qedhere
\end{proof}

\paragraph{Proof of Lemma \ref{LEM-LD}} From the definition of
$\Lambda_b(\cdot)$ and Lemma \ref{Rthx}, we have:
\begin{align*}
  \exp\left( \Lambda_b(\tnb)\right) &= \int_{-\infty}^{b}
  \exp(\tnb x) F(dx)\\
  &\leq 1 + c \tnb^{\kappa} + e^{2 \alpha }\bar{F}\left( \frac{2
      \alpha}{\tnb} \right) + \exp({\tnb}) \bar{F}(b) (1+ o(1)),
\end{align*}
for $\kappa \in ((\beta + \epsilon)^{-1}, \alpha).$ Usage of Lemma
\ref{Rthx} is justified because $b\tnb = -\log\left(\pn\right)
\nearrow \infty.$ The last term,
\[ \exp(\tnb b)\bar{F}(b) = \frac{1}{\pn}\bar{F}(b) = \frac{1}{n}.\]
From Lemma \ref{theta-rate}, we have $n \tnb^{\kappa} =
o\left(1\right)$ and $\bar{F}\left( 2 \alpha/\tnb \right) =
o\left({1}/{n}\right),$ uniformly for $b > \thr.$ Therefore,
\[ \exp \left( \Lambda_b(\theta_n)\right) \leq 1 +
\frac{1}{n}\left(1+o(1)\right), \text{ as } n \nearrow \infty. \]
\hfill\(\Box\)

\paragraph{Proof of Lemma \ref{NORM-TERM}}
Consider $\theta : \mathbb{R}^+ \rightarrow \mathbb{R}^+.$ From Lemma
\ref{Rthx}, we have that: for given $\epsilon > 0,$ if $x\theta(x)
\nearrow \infty,$ then there exists $x_{\epsilon}$ such that for all
$x > x_{\epsilon},$
\[ \int_{-\infty}^{x} e^{\theta(x) u}F(du) \leq 1 + c
\theta^{1+\delta}(x) + e^{2 \alpha }\bar{F}\left( \frac{2
    \alpha}{\theta(x)} \right) + e^{\theta(x)x} \bar{F}(x)
(1+\epsilon),\] for some $\delta > 0.$ 
By definition of $\theta_k(b)$ in \eqref{THETA-K}, we have $(b+ \nkm
\mu)\cdot\theta_k(b) \nearrow \infty,$ either if $b$ or $k$ grows to
infinity. Writing $\theta_k(b)$ as $\theta_k,$ for values of $b$ and
$k$ satisfying $b+\nkm \mu > x_\epsilon,$ we have,
\begin{align*}
  \exp\left( \Lambda_k(\theta_k)\right) &\leq 1 + c
  \theta_k^{1+\delta}+ e^{2 \alpha }\bar{F}\left( \frac{2
      \alpha}{\theta_k} \right) + e^{\theta_k \cdot (b+\nkm \mu)}
  \bar{F}(b+ \nkm \mu) (1+\epsilon)\\
  &\leq \exp\left( c\theta_k^{1+\delta} + e^{2 \alpha }\bar{F}\left(
      \frac{2 \alpha}{\theta_k} \right) + \frac{1}{n_k}(1+\epsilon)
  \right),
\end{align*}
because $1+x \leq e^x$ and $e^{\theta_k\cdot(b+\nkm
  \mu)}\bar{F}(b+\nkm\mu) = 1/n_k.$ Then,
\begin{align}
  \exp\left( n_k \Lambda_k(\theta_k)\right) \leq \exp\left(
    cn_k\theta_k^{1+\delta} + e^{2 \alpha }n_k\bar{F}\left( \frac{2
        \alpha}{\theta_k} \right) + 1+\epsilon \right).
  \label{INT-1}
\end{align}
Also see that,
\begin{align}
  n_k \theta_k^{1+\delta} &= \frac{n_k}{(b+\nkm
    \mu)^{1+\delta}}\left(\log\left(\frac{1}{n_k\bar{F}(b+\nkm
        \mu)}\right)\right)^{1+\delta} < \epsilon,
  \label{INT-2}
\end{align}
if $b$ and $k$ are such that $(b+ \nkm \mu)$ is large
enough. Similarly for given $\delta > 0,$ there exists $x_\delta$ such
that if $b + \nkm \mu > x_\delta,$ then
\begin{align*}
  \frac{\bar{F}\left(\frac{2\alpha}{\theta_k}\right)}{\bar{F}(b+\nkm
    \mu)} &= \frac{\bar{F}\left(\frac{2\alpha(b+\nkm
        \mu)}{-\log\left({n_k\bar{F}(b+ \nkm
            \mu)}\right)}\right)}{\bar{F}(b+\nkm \mu)}\\
  &\leq \left( \frac{1}{2\alpha}\log\left( \frac{1}{n_k\bar{F}(b+ \nkm
        \mu)}\right)\right)^{\alpha + \delta}.
\end{align*}
Then for values of $b$ and $k$ such that $(b + \nkm \mu)$ is large
enough,
\begin{align*}
  n_k\bar{F}\left(\frac{2\alpha}{\theta_k}\right) &\leq n_k\bar{F}(b+
  \nkm \mu) \left( \frac{1}{2\alpha}\log\left( \frac{1}{n_k\bar{F}(b+
        \nkm \mu)}\right)\right)^{\alpha +
    \delta}\\
  &= \frac{n_kL(b+\nkm \mu)}{(b+\nkm \mu)^{\alpha}}\left(
    \frac{1}{2\alpha}\log\left( \frac{1}{n_k\bar{F}(b+ \nkm
        \mu)}\right)\right)^{\alpha + \delta} < \epsilon,
  \end{align*}
  because $\alpha > 1.$ Combining this with \eqref{INT-1} and
  \eqref{INT-2}, for $b$ and $k$ such that $b+ \nkm \mu$ is
  sufficiently large,
  \[ \exp\left( n_k \Lambda_k(\theta_k)\right) \leq
  \exp(1+3\epsilon), \] thus establishing the claim. \hfill\(\Box\)

\paragraph{Proof of Lemma \ref{OTHER-TERM}}
Since $n_k = r\nkm,$
\begin{align*}
  \sup_k \frac{n_k\bar{F}(b+ \nkm \mu)}{q_k(b)} &= \sup_k
  \frac{n_k\bar{F}(b+\nkm \mu)}{\sum_{j=\nkm+1}^{n_k}\bar{F}(b+j\mu)}\\
  &\leq \frac{n_k\bar{F}(b+ \frac{n_k}{r}
    \mu)}{(1-r^{-1})n_k\bar{F}(b+n_k\mu)} < \infty,
\end{align*}
because of \eqref{LONG-TAIL-EXT}.  \hfill\(\Box\)\\

\paragraph{Proof of Lemma \ref{BCYC-NORM-TERM}} 
Since $\theta_b b \nearrow \infty$ and $\E X \neq 0,$ similar to Lemma
\ref{Rthx}, we have:
\begin{align*}
  \exp \left(\Lambda_b(\theta_b) \right) &\leq 1 + \theta_b \E X + c
  \theta_b^2 + \exp(2 \alpha) \bar{F} \left( \frac{2\alpha}{\theta_b}
  \right) +  \exp \left( \theta_b b \right) \bar{F}(b)(1+o(1))\\
  &= 1 - \theta_b \mu + c \theta_b^2 + \exp(2 \alpha) \bar{F} \left(
    \frac{2\alpha}{\theta_b} \right) + \frac{1}{b}(1+o(1)).
\end{align*}
It follows from the definition of $\theta_b$ and a simple application
of \eqref{LONG-TAIL-EXT} that
\[ \frac{1}{b} = o \left(\theta_b \right) \text{ and } \bar{F} \left(
  \frac{2\alpha}{\theta_b} \right) = o \left(\theta_b \right), \text{
  as } b \nearrow \infty.\]
Therefore,
\[ \exp \left(\Lambda_b(\theta_b) \right) \leq 1 - \theta_b \mu
(1+o(1)),\] as $b \nearrow \infty.$ Then 
\[ \varlimsup_{b \rightarrow \infty} \sup_{ n \geq 1} \exp \left(n
  \Lambda_b(\theta_b) \right) \leq \inf_y \sup_n \sup_{b > y}
(1-\theta_b \mu (1+o(1)))^n \leq 1,\] which proves the
claim. \hfill$\Box$

\bibliographystyle{acmtrans-ims}
\bibliography{IS_RV_SI}

\end{document}